\documentclass[reqno]{amsart}
\usepackage{amssymb}
\usepackage{graphicx}

\usepackage[usenames, dvipsnames]{color}
\usepackage{verbatim}
\usepackage{mathrsfs}
\usepackage{bm}
\usepackage{cite}
\usepackage{bbm}

\numberwithin{equation}{section}

\newtheorem{theorem}{Theorem}[section]
\newtheorem{corollary}[theorem]{Corollary}
\newtheorem{lemma}[theorem]{Lemma}
\newtheorem{prop}[theorem]{Proposition}

\theoremstyle{definition}
\newtheorem{remark}[theorem]{Remark}

\theoremstyle{definition}
\newtheorem{definition}[theorem]{Definition}

\theoremstyle{definition}

\makeatletter
\def\dashint{\operatorname%
{\,\,\text{\bf-}\kern-.98em\DOTSI\intop\ilimits@\!\!}}
\makeatother

\def\\det{\text{\det}}

\def\Xint#1{\mathchoice
 {\XXint\displaystyle\textstyle{#1}}%
 {\XXint\textstyle\scriptstyle{#1}}%
 {\XXint\scriptstyle\scriptscriptstyle{#1}}%
 {\XXint\scriptscriptstyle\scriptscriptstyle{#1}}%
 \!\int}
\def\XXint#1#2#3{{\setbox0=\hbox{$#1{#2#3}{\int}$}
  \vcenter{\hbox{$#2#3$}}\kern-.5\wd0}}

\def\dashint{\Xint-}

\def\.5{\frac{1}{2}}

\newcommand{\RN}[1]{%
  \textup{\uppercase\expandafter{\romannumeral#1}}%
}

\renewcommand{\epsilon}{\varepsilon}

\newcounter{marnote}

\begin{document}

\title[Regularity for degenerate parabolic equations]{Regularity for solutions to inhomogeneous degenerate parabolic $p$-Laplace equations}

\author[X. Hao]{Xia Hao}
\address[X. Hao] {School of Mathematical Sciences, Hebei Normal University, Shijiazhuang 050024, China.}
\email{xiahao0915@163.com }

\author[Y. Li]{Yan Li}
\address[Y. Li] {School of Mathematical Sciences, Beijing Normal University, Beijing 100875, China. }
\email{yanli25@mail.bnu.edu.cn}

\author[Z.W. Zhao]{Zhiwen Zhao}

\address[Z.W. Zhao]{School of Mathematics and Physics, University of Science and Technology Beijing, Beijing 100083, China.}
\email{zwzhao365@163.com;\,zwzhao365@gmail.com}

%\footnote{}

\date{\today} % delete this line to display the current date

%%% BEGIN DOCUMENT

\maketitle
%\tableofcontents
\begin{abstract}
 For weak solutions to quasilinear degenerate parabolic equations of $p$-Laplace type, a central obstacle in applying the method of intrinsic scaling to prove their H\"{o}lder regularity is the derivation of forward-in-time propagation estimate for the spatial measure of level sets.\ In this paper, we revisit and overcome this difficulty for an inhomogeneous degenerate parabolic $p$-Laplace equation with bounded nonnegative time-independent forcing term and initial data.\ Our new proof is built upon the temporal Lipschitz estimate that we establish on any time interval separated from the initial time by a positive amount.\ The strong regularity gain furnished by this estimate allows us to derive the forward propagation of level-set measures directly, thereby bypassing the energy inequality that served as the starting point of the previous approach.\ Furthermore, we establish space-time $W^{1,\infty}_{x,t}$ estimates away from the initial time and construct counterexamples to demonstrate the sharpness of the positive waiting-time condition under merely bounded initial data.

\end{abstract}

\maketitle
%\date{}
%\maketitle
%{\bf Abstract}

%{\bf{Keywords:}} H\"{o}lder estimates; degenerate parabolic $p$-Laplace equations; temporal Lipschitz regularity
%
%\noindent{\bf{MSC numbers}}: {35K92; 35B51; 35B65.}

\section{Introduction}

Consider a smooth bounded domain $\Omega\subset \mathbb{R}^{n}$ with $n\geq2$.\ For any fixed $T>0$, let $\Omega_{T}=\Omega\times(-T,0]$, and denote its standard parabolic boundary by $\partial_{pa}\Omega_{T}$.\ For $p>1$, we study the following initial-boundary value problem for the inhomogeneous parabolic $p$-Laplace equation:
\begin{align}\label{PO001}
\begin{cases}
\partial_{t}u-\mathrm{div}(|\nabla u|^{p-2}\nabla u)=f(x),& \mathrm{in}\;\Omega_{T},\\
|\nabla u|^{p-2}\nabla u\cdot\nu=0,&\mathrm{on}\;\partial\Omega\times(-T,0],\\
u(x,-T)=\phi(x),&\mathrm{in}\;\Omega,
\end{cases}
\end{align}
where $\nu$ denotes the outward unit normal vector to $\partial\Omega$, $f$ and $\phi$ are two nonnegative functions satisfying 
\begin{align}\label{AQ000}
\|f\|_{L^{\infty}(\Omega_{T})}=\|f\|_{L^{\infty}(\Omega)}<\infty,\quad\|\phi\|_{L^{\infty}(\Omega)}<\infty.
\end{align}
The parabolic $p$-Laplace equation has received substantial attention because of both its intrinsic mathematical interest and its relevance to various applications, including nonlinear porous-media flow, image processing, and game theory, see e.g. \cite{DLK2013,ACM2004,BV2004,PS2008}.\ In particular, from the physical viewpoint of diffusion, the exponent $p$ gives rise to three distinct regimes.\ For $1<p<2$, the equation exhibits fast diffusion, and solutions may undergo finite-time extinction under suitable boundary conditions.\ By contrast, for $p>2$, the equation exhibits slow diffusion, with convergence toward equilibrium typically occurring at an algebraic rate.\ At the borderline $p=2$, the equation reduces to the classical heat equation, for which convergence to equilibrium is exponential.

From the analytical perspective, the equation exhibits equally distinct behaviors according to the value of $p$.\ When $1<p<2$, the equation is singular, as the modulus of ellipticity $|\nabla u|^{p-2}$ blows up near points where the gradient vanishes. When $p>2$, however, the equation is degenerate because $|\nabla u|^{p-2}$ vanishes at points where the gradient is zero.\ The singular and degenerate nature of the parabolic $p$-Laplace equation constitutes an obstruction to the full smoothness available when $p=2$ and generally confines solutions to the regularity class $C^{1,\alpha}$ for some $0<\alpha<1$. Motivated by the need to understand the impact of singular and degenerate mechanisms on solution regularity, extensive research has been conducted on the regularity theory for quasilinear elliptic and parabolic equations, see e.g.  \cite{BDGLS2026,CM2011,CM2014,L2019,D1986,CD1988,DZ2022,DZ2024,D1983,D1993,DM2010,DM201002,DM2011,U2008,E1982,DK1992,L1994,DGV2012,DGV2008,DGV200802,AL1983,DF198501,DF198502,DF1984,C1991,BSS2022}.\ Related results have also been established for more general doubly nonlinear parabolic equations, see e.g.  \cite{I1989,I199401,I199402,I1995,BDMS2018,BDL2021,BDGLS202301,BDGLS202302,KSU2012,PV1993,VV2022,BFV2018,V2007,DK2007,DKV1991,JX2019,JRX2023,K1988}.

In the stationary setting, the local H\"{o}lder continuity of weak solutions to the $p$-Laplace equation can be established by the classical De Giorgi truncation method created in \cite{D1957}.\ For the heat equation, the quadratic homogeneity of the parabolic energy estimates is compatible with the Sobolev embedding, allowing a natural extension of the elliptic De Giorgi scheme.\ When $p\neq2$, however, the quadratic energy term arising from the time derivative and the $p$-growth spatial term no longer have matching homogeneity, posing substantial difficulties in implementing the same approach.\ In fact, the parabolic De Giorgi truncation scheme is structured around three successive and interconnected estimates: a forward-in-time propagation estimate for the spatial distribution, a space-time level-set measure decay estimate, and a measure-to-point oscillation reduction estimate.\ In the degenerate range $p>2$, the principal difficulty lies in achieving the required propagation estimate, whereas in the singular range $1<p<2$, the main obstacle is to establish the decay estimate. 

A major breakthrough in the degenerate case was achieved by DiBenedetto \cite{D1986,D1993}, who introduced the intrinsic-scaling technique to restore the balance in the energy estimates by adapting the space-time geometry to the oscillation of the solution.\ Subsequently, DiBenedetto, Gianazza, and Vespri \cite{DGV2008,DGV2012} further refined the intrinsic-scaling framework through an exponential change of variables, which allows the relevant measure information to be propagated over sufficiently long intrinsic time intervals and thereby overcomes the obstruction associated with the expansion in time.\ This approach led to the local H\"{o}lder continuity of weak solutions to general degenerate quasilinear parabolic equations.

To address the difficulty with the decay estimates in the singular range $1<p<2$, Hao, Li and Zhao \cite{HLZ2026} introduced a strategy based on establishing
temporal Lipschitz regularity away from the initial time, thereby successfully reducing the parabolic space-time distribution decay estimate to an elliptic-type decay estimate of spatial distribution on almost every fixed time slice.\ For this purpose, they developed a time-rescaling technique coupled with $L^\infty$ comparison estimates to refine the method of intrinsic scaling in the singular regime.

In this paper, we revisit the difficulty of establishing the forward level-set measure propagation estimate and overcome it for an inhomogeneous parabolic $p$-Laplace equation with any $p>2$.\ Rather than initiating the argument from exponential change of variables as in \cite{DGV2008}, we adopt a different route based on improved temporal regularity.\ More precisely, we first establish a time-Lipschitz estimate on arbitrary intervals staying away from the initial time.\ The additional regularity obtained in this way provides sufficient control to derive the propagation of spatial level-set measures directly, thereby avoiding the lengthy technical estimates and complicated calculations required in the previous approach.\ By exploiting the time-Lipschitz estimate, we further establish space-time $W_{x,t}^{1,\infty}$ estimates after any prescribed positive elapsed time for both Neumann and Dirichlet boundary conditions.\ We also construct some counterexamples to show the sharpness in the sense that a positive waiting time is indispensable for such estimates based solely on bounded initial data.

\begin{definition}
A function $u\in C([-T,0];L^{2}(\Omega))\cap L^{p}((-T,0);W^{1,p}(\Omega))$ is said to be a weak solution to problem \eqref{PO001} provided that, for any
$-T\leq t_{1}<t_{2}\leq0$,
\begin{align*}
	&\int_{\Omega}u\varphi\,dx\Big|_{t_{1}}^{t_{2}}+\int_{t_{1}}^{t_{2}}\int_{\Omega}\left(-u\partial_{t}\varphi
	+|\nabla u|^{p-2}\nabla u\cdot\nabla\varphi\right)\,dxdt=\int_{t_{1}}^{t_{2}}\int_{\Omega}f(x)\varphi\,dxdt,
\end{align*}
for every $\varphi\in W^{1,2}((-T,0);L^{2}(\Omega))\cap L^{p}((-T,0);W^{1,p}(\Omega))$.
\end{definition}
Throughout this paper, unless otherwise specified, we use $C$ to represent a positive constant depending only on the prescribed data, including $n,\,p,\,\|\phi\|_{L^{\infty}(\Omega)}$ and $\|f\|_{L^{\infty}(\Omega)}$, and whose value may vary from one occurrence to another.\ We now state the main results as follows.
\begin{theorem}\label{ZWTHM90}
Let $p>2$ and $n\geq2$.\ For every weak solution $u$ of \eqref{PO001} in $B_{1}\times(-1,0]$, there exists an exponent $\alpha\in(0,1)$, depending only on the prescribed data, such that for all $(x,t),(y,s)\in B_{\frac{5}{8}}\times(-\frac{1}{2},0]$ with $(x,t)\neq(y,s)$,
\begin{align}\label{QNAW001ZW}
\frac{|u(x,t)-u(y,s)|}{|x-y|^{\alpha}+|t-s|}\leq C.
\end{align}

\end{theorem}

\begin{remark}
For a general space-time domain $\Omega_T$, let $D\subset\subset\Omega$ be any fixed compact subdomain.\ Denote $R_{0}:=\min\{1,\mathrm{dist}(D,\partial\Omega),T^{1/p}\}$.\ Proceeding as in the proof of Theorem \ref{ZWTHM90} with a slight modification, we obtain the H\"{o}lder estimate in \eqref{QNAW001ZW} on
$D\times(-T/2,0]$ in place of $B_{5/8}\times(-1/2,0]$.\ Here $C$ may depend on $R_{0}$, while $\alpha$ remains independent of $R_{0}$.

\end{remark}

\begin{remark}\label{RE01}
The same argument implies that the interior estimate in Theorem \ref{ZWTHM90} remains valid when the Neumann condition in \eqref{PO001} is replaced by the Dirichlet condition
\begin{align}\label{AQ26}
	u=\phi,\quad\mathrm{on}\;\partial_{pa}(B_{1}\times(-1,0]),\quad \text{with}\;\phi\geq0, \;\phi\in L^{\infty}.
\end{align}
We next present a counterexample showing that uniform $L^\infty$ bounds on smooth Dirichlet data do not yield uniform global H\"{o}lder and $W^{1,\infty}$ estimates.\ Fix $p>2$ and $n=2$. For $0<\varepsilon<\frac{1}{4}$, define
	\begin{align*}
		\Psi_\varepsilon(x)=\frac\pi2+\arctan\frac{x_2}{1+\varepsilon-x_1},\quad w_\varepsilon=\frac{\Psi_\varepsilon}{\|\Psi_\varepsilon\|_{L^\infty(B_1)}}.
	\end{align*}
	Then $\pi/2<\|\Psi_\varepsilon\|_{L^\infty(B_1)}<\pi$, $\|w_\varepsilon\|_\infty=1$, and $w_\varepsilon$ is positive and smooth on a neighborhood of $\overline{B}_{1}$.\ Denote $r_\varepsilon^{2}=(1+\varepsilon-x_1)^2+x_{2}^{2}$. It then follows from a direct computation that
	\begin{align*}
		\nabla\Psi_\varepsilon=\frac{(x_2,1+\varepsilon-x_1)}{r_{\varepsilon}^{2}},\quad|\nabla w_{\varepsilon}|^{p-2}\nabla w_\varepsilon=\frac{(x_2,1+\varepsilon-x_{1})}{r_{\varepsilon}^{p}\|\Psi_\varepsilon\|_{L^\infty(B_1)}^{p-1}},
	\end{align*}
which leads to that $-\mathrm{div}(|\nabla w_{\varepsilon}|^{p-2}\nabla w_\varepsilon)=0$ in $B_{1}$. Consequently, $u_\varepsilon(x,t)=w_\varepsilon(x)$, $f_\varepsilon=0$ and $\phi_\varepsilon=w_\varepsilon$ solve
\begin{align*}
	\begin{cases}
		\partial_{t}u_{\varepsilon}-\mathrm{div}(|\nabla u_{\varepsilon}|^{p-2}\nabla u_{\varepsilon})=0,& \mathrm{in}\;B_{1}\times(-1,0],\\
	u_{\varepsilon}=\phi_{\varepsilon}(x),&\mathrm{on}\;\partial_{pa}(B_{1}\times(-1,0]),
	\end{cases}
\end{align*}	
satisfying that	$\|u_\varepsilon\|_{\infty}=\|\phi_{\varepsilon}\|_{\infty}=1.$ For $0<\varepsilon<\frac{1}{4}$, choose
	\begin{align*}
		x_\varepsilon=(1-\varepsilon,2\varepsilon),\quad y_\varepsilon=(1-\varepsilon,-2\varepsilon).
	\end{align*}
Both points are in $B_1$, because their squared norms are $1-2\varepsilon+5\varepsilon^2<1$.\ Observe that
	\begin{align*}
		|x_\varepsilon-y_\varepsilon|=4\varepsilon,\quad
		|w_\varepsilon(x_\varepsilon)-w_\varepsilon(y_\varepsilon)|
		=\frac\pi{2\|\Psi_\varepsilon\|_{L^\infty(B_1)}}\geq\frac{1}{2}.
	\end{align*}
	At the common time $t=s=-\frac{1}{8}$, any uniform estimate
	\begin{align*}
		|u(x,t)-u(y,s)|\leq C\bigl(|x-y|^\alpha+|t-s|\bigr),
		\quad (x,t),(y,s)\in B_{\frac{5}{8}}\times\Big(-\frac{1}{2},0\Big],
	\end{align*}
	with $C<\infty$ and $\alpha>0$ depending only on $n$,\ $p$,\ $\|\phi\|_{L^{\infty}(B_{1})}$ and $\|f\|_{L^{\infty}(B_{1})}$, would imply $1/2\leq C(4\varepsilon)^\alpha$, which is a contradiction as $\varepsilon\rightarrow0$. Furthermore, we have
	\begin{align*}
		|\nabla u_{\varepsilon}(1-\varepsilon,0)|=|\nabla w_{\varepsilon}(1-\varepsilon,0)|=\frac{1}{2\|\Psi_\varepsilon\|_{L^\infty(B_1)}\varepsilon}
		\geq\frac1{2\pi\varepsilon}.
	\end{align*}
So no uniform global gradient bound follows from those $L^{\infty}$ data even with $u_t=0$.

\end{remark}	

\begin{theorem}\label{ZWTHM96}
Let $p\in(1,2)\cup(2,\infty)$, $n\geq2$ and \eqref{AQ000} hold.\ Consider the equation
\begin{align*}
	\partial_{t}u-\mathrm{div}(|\nabla u|^{p-2}\nabla u)=f(x),\quad \text{in}\; B_{1}\times(-1,0]. 
\end{align*}
For every fixed $0<\delta\ll 1$, the following estimates hold:

$(i)$ if the homogeneous Neumann condition in \eqref{PO001} is satisfied, then
\begin{align}\label{Neu01}
	&\|u\|_{W^{1,\infty}(B_{1}\times(-1+\delta,0))}\notag\\
	&\leq C \bigg(\frac{\|f\|_{L^{\infty}(B_{1})}+\|\phi\|_{L^{\infty}(B_{1})}}{\delta|p-2|}+\|f\|_{L^{\infty}(B_{1})}\bigg)^{\frac{1}{p-1}};
\end{align}

$(ii)$ if the inhomogeneous Dirichlet condition in \eqref{AQ26} is satisfied, assume in addition that $\phi\in C^{1,1}(\overline{B}_{1})$, then
\begin{align}\label{Dir02}
	&\|u\|_{W^{1,\infty}(B_{1}\times(-1+\delta,0))}\notag\\
	&\leq C\Bigg(\|\phi\|_{C^{1,1}(\overline{B}_{1})}+\bigg(\frac{\|f\|_{L^{\infty}(B_{1})}+\|\phi\|_{L^{\infty}(B_{1})}}{\delta|p-2|}+\|f\|_{L^{\infty}(B_{1})}\bigg)^{\max\{1,\frac{1}{p-1}\}}\Bigg),
\end{align}
where $C=C(n,p).$	
	
\end{theorem}
\begin{remark}
By the same way, we obtain that the estimate in \eqref{Neu01} also holds if the condition in \eqref{AQ26} is replaced by
\begin{align*}
	u=0,\quad \mathrm{on}\;\partial B_{1}\times(-1,0],\qquad
	u(x,-1)=\phi(x),\quad \mathrm{in}\;B_{1},
\end{align*}
where $\phi\geq0$ and $\phi\in L^{\infty}(B_{1})$.\ Combined with the counterexamples presented in Remark \ref{RE01}, this reveals a pronounced difference between homogeneous and inhomogeneous boundary conditions in their effects on solutions.\ In particular, homogeneous Neumann and Dirichlet conditions permit global uniform estimates, while merely bounded inhomogeneous boundary data may prevent such estimates, as demonstrated by the counterexamples.
\end{remark}	

\begin{remark}
It is worth noting that the upper bounds in \eqref{Neu01}--\eqref{Dir02} blow up as $\delta\rightarrow0$, $p\rightarrow1$, or $p\rightarrow2$. The failure to include $p=2$ is caused by the time-Lipschitz estimate proved in Theorem \ref{thm10} below.
\end{remark}

\begin{remark}
In contrast to Theorem 2.8 of \cite{KL2000}, where the regularity $L^{p+\varepsilon}_{t}W^{1,p+\varepsilon}$ is obtained for some $\varepsilon>0$, our setting allows us to establish the stronger regularity $W^{1,\infty}_{x,t}$ in Theorem \ref{ZWTHM96}.\ Moreover, our results not only cover the full singular and degenerate ranges of the exponent $p$, but also explicitly quantify the dependence of the estimates on the prescribed data.
	
\end{remark}	

\begin{remark}[An explicit concentrating family for $p>2$]
	Fix $n\geq2$ and $p>2$. Denote
	\begin{align*}
		&\Lambda=n(p-2)+p,\quad q=\frac p{p-1},\quad
		\beta=\frac{p-1}{p-2},\quad c=\frac{p-2}{p}\Lambda^{-1/(p-1)},\quad R_*=c^{-1/q}.
	\end{align*}
	Consider the following Barenblatt profile
	\begin{align*}
		\mathcal B(y,\tau)=\tau^{-n/\Lambda}
		\big(\max\big\{1-c|y|^q\tau^{-q/\Lambda},0\big\}\big)^{\beta},
		\quad (y,\tau)\in\mathbb R^{n}\times(0,\infty).
	\end{align*}
By a direct calculation, we have 
	\begin{align*}
		|\nabla_y\mathcal B|^{p-2}\nabla_y\mathcal {B}=-\frac{y}{\Lambda\tau}\mathcal {B},\quad
		\partial_\tau\mathcal B=-\frac{n\mathcal B+y\cdot\nabla_y\mathcal B}{\Lambda\tau},
	\end{align*}
which gives that $\partial_{\tau}\mathcal B-\mathrm{div}_{y}(|\nabla_y\mathcal B|^{p-2}\nabla_y\mathcal {B})=0$ in $\mathbb{R}^{n}\times(0,\infty)$.\ Indeed, $\mathcal B$ is $C^1$ in $(y,\tau)$ for $\tau>0$, and the explicitly computed flux $-y\mathcal{B}/(\Lambda\tau)$ is also $C^1$.\ At the moving interface $|y|=R_{*}\tau^{1/\Lambda}$, the solution, its first derivatives, and its flux vanish. Hence the equation holds distributionally across the interface without an additional interface measure. 
	
For small $\varepsilon>0$ and $-1\leq t\leq0$, let
\begin{align*}
		u_\varepsilon(x,t)=\mathcal{B}\Big(\frac{x}{\varepsilon},
		1+\frac{t+1}{\varepsilon^p}\Big),\quad
		\phi_\varepsilon(x)=\mathcal{B} (x/\varepsilon,1),\quad f_{\varepsilon}=0.
\end{align*}
Pick the value of $\varepsilon$ sufficiently small such that the support radius satisfies
\begin{align*}
	R_\varepsilon(t)&=\varepsilon R_{*}\left(1+\frac{t+1}{\varepsilon^p}\right)^{1/\Lambda}\leq 2^{1/\Lambda}R_*\varepsilon^{1-p/\Lambda}<\frac{1}{2}.
\end{align*}
Then $u_{\varepsilon}$ solves
\begin{align*}
\begin{cases}
	\partial_{t}u_{\varepsilon}-\mathrm{div}_x\big(|\nabla_{x}u_{\varepsilon}|^{p-2}\nabla_{x}u_{\varepsilon}\big)=0,&\mathrm{in}\;B_{1}\times(-1,0],\\
	u_{\varepsilon}(x,t)=\phi_{\varepsilon}(x)=0,&\mathrm{on}\;\partial B_{1}\times(-1,0],\\
	|\nabla_{x}u_{\varepsilon}|^{p-2}\nabla_{x}u_{\varepsilon}\cdot\nu=0,&\mathrm{on}\;\partial B_{1}\times(-1,0],\\
	u_{\varepsilon}(x,-1)=\phi_{\varepsilon}(x),&\mathrm{in}\; B_{1},
\end{cases}		
\end{align*}	
where $\phi_{\varepsilon}\in C_{0}^{1}(B_{1/2})\cap W^{1,p}(B_{1})$, and
	\begin{align*}
 0\leq u_{\varepsilon}\leq1,\quad \|\phi_{\varepsilon}\|_{L^{\infty}(B_{1})}=\|u_{\varepsilon}\|_{L^{\infty}(B_{1}\times(-1,0])}=1.
	\end{align*}
	
At the strictly positive elapsed time $t_\varepsilon=-1+\varepsilon^p$, take
\begin{align*}
x_{\varepsilon}=0,\quad y_\varepsilon=\varepsilon R_{*}2^{1/\Lambda}e_{1},
\end{align*}
where $e_{1}=(1,0,...,0).$ Both points lie in $B_{1/2}$ for small $\varepsilon$, and
	\begin{align*}
		|x_\varepsilon-y_\varepsilon|
		=\varepsilon R_{*}2^{1/\Lambda},\quad
		|u_\varepsilon(x_\varepsilon,t_\varepsilon)
		-u_\varepsilon(y_\varepsilon,t_\varepsilon)|=2^{-n/\Lambda}.
	\end{align*}
	Consequently, no constants $C<\infty$ and $\alpha>0$ depending only on $n,p$ and the displayed bounded data can give
	\begin{align*}
		|u(x,t)-u(y,s)|\leq C\bigl(|x-y|^\alpha+|t-s|\bigr),\quad(x,t),(y,s)\in B_1\times(-1,0),
	\end{align*}
since it contradicts such a bound as $\varepsilon\rightarrow0$ by picking $t=s=t_\varepsilon$.
	
The blow-up also appears directly in the derivative norms. Define
\begin{align*}
	x_{\ast}=2^{1/\Lambda}(2c)^{-1/q}e_{1}.
\end{align*}
Then we have
	\begin{align*}
		|\nabla_xu_\varepsilon(\varepsilon x_{\ast},t_{\varepsilon})|=\kappa_{1}\varepsilon^{-1},\quad
		|\partial_{t}u_\varepsilon(0,t_\varepsilon)|=\varepsilon^{-p}|\partial_{\tau}\mathcal{B} (0,2)|=\kappa_{2}\varepsilon^{-p},
	\end{align*}
where $\kappa_{1}=2^{-(n+1)/\Lambda-\beta+1}\Lambda^{-1/(p-1)}(2c)^{-(q-1)/q}$ and $\kappa_{2}=n\Lambda^{-1}2^{-n/\Lambda-1}$.\ This implies that 
	\begin{align*}
		\|\nabla u_\varepsilon\|_{L^\infty(B_{1}\times(-1,0])}\geq \kappa_{1}\varepsilon^{-1}\rightarrow\infty,
		\quad
		\|\partial_{t}u_{\varepsilon}\|_{L^\infty(B_{1}\times(-1,0])}\geq \kappa_{2}\varepsilon^{-p}\rightarrow\infty,
	\end{align*}
as $\varepsilon\rightarrow0$.\ These are initial-layer failures of \emph{uniform estimates}.\ In particular, this family does not admit a uniform full-cylinder $W^{1,\infty}_{x,t}$ bound for the solution, which also shows that the time-Lipschitz estimate in Theorem \ref{thm10} below cannot be extended uniformly up to the initial time.
\end{remark}

\begin{remark}[A concentrating family for $1<p<2$]
	Fix $n\geq2$ and $1<p<2$. Let
	\begin{align*}
		q=\frac p{p-1}>2,\quad
		\Phi(y)=\big(\max\{1-|y|^{q},0\}\big)^q,\quad
		\phi_{\varepsilon}(x)=\Phi(x/\varepsilon),\quad 0<\varepsilon<\frac{1}{4}.
	\end{align*}
	Then we have $\phi_{\varepsilon}\in C_{0}^2(B_{1/2})$, $0\leq\phi_\varepsilon\leq1$
	and $\|\phi_\varepsilon\|_\infty=1$. For $0<|y|<1$,
	\begin{align*}
		|\nabla\Phi|^{p-2}\nabla\Phi&=-q^{2(p-1)} |y|^{(q-1)(p-2)+q-2}y(1-|y|^q)^{(q-1)(p-1)}=-q^{2(p-1)} y(1-|y|^{q}).
	\end{align*}
The flux $-q^{2(p-1)} y\max\{1-|y|^q,0\}$ is a Lipschitz vector field vanishing on $|y|=1$ and its divergence has no surface-measure term. Consequently, in the distributional sense, we have
	\begin{align*}
	\begin{cases}	
		g(y):=\mathrm{div}_y\bigl(|\nabla\Phi|^{p-2}\nabla\Phi\bigr)
		=-q^{2(p-1)}\big[n-(n+q)|y|^q\big]\chi_{\{|y|<1\}},\\
		\mathrm{div}_x\big(|\nabla\phi_\varepsilon|^{p-2}
		\nabla\phi_\varepsilon\big)=\varepsilon^{-p}g(x/\varepsilon),
		\quad \|g\|_\infty=K:=q^{2(p-1)}\max\{n,q\},
		\end{cases}
	\end{align*}
where $\chi_{\{|y|<1\}}$ is the indicator function of the unit ball. Let $u_\varepsilon\in C([-1,0];L^{2}(B_{1}))\cap L^{p}((-1,0);W^{1,p}(B_{1}))$ be the unique energy solution of
	\begin{align*}
		\partial_{t}u_{\varepsilon}-\mathrm{div}_{x}\big(|\nabla u_\varepsilon|^{p-2}\nabla u_\varepsilon\big)=0
		\quad\text{in }B_{1}\times(-1,0],\quad
		u_{\varepsilon}(x,-1)=\phi_{\varepsilon}(x),\quad\mathrm{in}\; B_{1}
	\end{align*}
	with either $u_\varepsilon=0$ or $|\nabla u_\varepsilon|^{p-2}\nabla u_\varepsilon\cdot\nu=0$ on $\partial B_{1}\times(-1,0]$.\ The construction applies separately to these two boundary conditions. By comparison principle, we have $0\leq u_{\varepsilon}\leq1$. The functions $\phi_\varepsilon\pm K\varepsilon^{-p}(t+1)$ are, respectively, a supersolution and a subsolution. Indeed, substitution in the equation yields the following residuals
	\begin{align*}
		-K\varepsilon^{-p}-\varepsilon^{-p}g(x/\varepsilon)\leq0,\quad K\varepsilon^{-p}-\varepsilon^{-p}g(x/\varepsilon)\geq0.
	\end{align*}
Their Neumann flux is zero.\ In the Dirichlet case their lateral values are, respectively, at most and at least zero.\ Comparison, followed by the same comparison for time translates, gives that for $-1\leq s<t\leq0$,
	\begin{align}\label{COM01}
		\|u_{\varepsilon}(t)-\phi_{\varepsilon}\|_{L^{\infty}(B_{1})}\leq K\varepsilon^{-p}(t+1),\quad \|u_{\varepsilon}(t)-u_{\varepsilon}(s)\|_{L^{\infty}(B_{1})}\leq K\varepsilon^{-p}(t-s),
	\end{align}
which reads that $\|u_{\varepsilon}\|_{L^{\infty}(B_{1}\times(-1,0])}=1$ and $\|\partial_{t}u_{\varepsilon}\|_{L^{\infty}(B_{1}\times(-1,0])}\leq K\varepsilon^{-p}$. 
	
Denote
	\begin{align*}
		a=\bigg[1-\Big(\frac{3}{4}\Big)^{1/q}\bigg]^{1/q}\in(0,1),
		\quad b=\frac1{8K},\quad e_{1}=(1,0,...,0).
	\end{align*}
	Since $\Phi(ae_1)=3/4$ and $\Phi(e_1)=0$, it follows from \eqref{COM01} that for $-1<t<-1+b\varepsilon^{p}$,
	\begin{align}\label{QM01}
		u_{\varepsilon}(a\varepsilon e_1,t)\geq\frac{5}{8},\quad
		u_{\varepsilon}(\varepsilon e_1,t)\leq\frac{1}{8},\quad
		u_{\varepsilon}(a\varepsilon e_1,t)-u_{\varepsilon}(\varepsilon e_{1},t)
		\geq\frac{1}{2}.
	\end{align}
	Then for any $\alpha\in(0,1)$, the spatial H\"{o}lder quotient is at least
	$[2(1-a)^\alpha]^{-1}\varepsilon^{-\alpha}$, and
	\begin{align}\label{AM20}
		\|\nabla u_{\varepsilon}\|_{L^\infty(B_{1}\times(-1,0])}
		\geq\frac{1}{2(1-a)}\varepsilon^{-1}.
	\end{align}
	
We next establish a temporal lower bound.\ By uniqueness, $u_\varepsilon$ is radial.\ Write $u_\varepsilon(x,t)=U_\varepsilon(|x|,t)$.\ For almost every $-1<t<-1+b\varepsilon^{p}$, it follows from \eqref{QM01} and H\"{o}lder's inequality that
	\begin{align*}
		\frac{1}{2}
		\leq&\int_{a\varepsilon}^{\varepsilon}|\partial_rU_{\varepsilon}|\,dr
		\leq\left(\int_{a\varepsilon}^{\varepsilon}|\partial_{r}U_{\varepsilon}|^{p} r^{n-1}\,dr\right)^{1/p}\left(\int_{a\varepsilon}^{\varepsilon}
		r^{-(n-1)/(p-1)}\,dr\right)^{(p-1)/p}\notag\\
		\leq&\left(\int_{a\varepsilon}^{\varepsilon}|\partial_{r}U_{\varepsilon}|^{p} r^{n-1}\,dr\right)^{1/p}\varepsilon^{\frac{p-n}{p}}J^{\frac{p-1}{p}},\quad J:=\frac{p-1}{p-n}\left(1-a^{\frac{p-n}{p-1}}\right).
	\end{align*}
Then we obtain
	\begin{align}\label{AF01}
		\int_{B_{1}}|\nabla u_\varepsilon(x,t)|^p\,dx
		\geq E\varepsilon^{n-p},\quad
		E:=|\mathbb S^{n-1}|\,2^{-p}J^{1-p}.
	\end{align}
	
From mass conservation in the Neumann case and mass decay in the Dirichlet case, we deduce that for a.e. $-1<t<-1+b\varepsilon^{p}$,
	\begin{align}\label{AF001}
		\int_{B_{1}}u_{\varepsilon}(x,t)\,dx\leq M\varepsilon^{n},\quad M:=\int_{B_1}\Phi(y)\,dy>0.
	\end{align}
For the Neumann condition this is an equality, obtained with the spatial test function $1$.\ For the Dirichlet condition, using the zero-trace test
	$\min\{u_\varepsilon/\eta,1\}$ with $\eta\in(0,1)$, we have
	\begin{align*}
		\int_{B_{1}} H_\eta(u_\varepsilon(t))\,dx\leq&\int_{B_{1}} H_\eta(u_\varepsilon(t))\,dx+\frac{1}{\eta}\int_{0}^{t}\int_{B_{1}\cap\{0<u_{\varepsilon}<\eta\}}|\nabla u_{\varepsilon}|^{p}dxds\notag\\
		\leq&\int_{B_{1}} H_\eta(\phi_\varepsilon)\,dx,\quad
		H_\eta(s)=\int_0^s\min\{r/\eta,1\}\,dr,\quad s\geq0.
	\end{align*}
Then sending $\eta\rightarrow0$, we prove \eqref{AF001}. Taking $t_{1}=-1+\frac{b}{2}\varepsilon^{p}$ and
	$t_{2}=-1+b\varepsilon^{p}$, we obtain from the energy identity and \eqref{AF01}--\eqref{AF001} that
	\begin{align*}
		bE\varepsilon^n
		&\leq 2\int_{t_1}^{t_2}\int_{B_{1}}|\nabla u_\varepsilon|^{p}\,dxdt
		=\|u_{\varepsilon}(\cdot,t_{1})\|_{L^{2}(B_{1})}^{2}-\|u_{\varepsilon}(\cdot,t_{2})\|_{L^{2}(B_{1})}^{2}\\
		&=\int_{B_{1}}\big(u_{\varepsilon}(\cdot,t_{1})-u_\varepsilon(\cdot,t_{2})\bigr)
		\big(u_{\varepsilon}(\cdot,t_{1})+u_{\varepsilon}(\cdot,t_{2})\big)\,dx\\
		&\leq 2M\varepsilon^{n}
		\|u_{\varepsilon}(\cdot,t_{2})-u_{\varepsilon}(\cdot,t_{1})\|_{L^{\infty}(B_{1})},
	\end{align*}
which shows that
	\begin{align}\label{AM21}
		\frac EM\varepsilon^{-p}\leq\frac{\|u_{\varepsilon}(\cdot,t_{2})-u_{\varepsilon}(\cdot,t_{1})\|_{L^{\infty}(B_{1})}}{t_{2}-t_{1}}\leq\|\partial_{t}u_{\varepsilon}\|_{L^{\infty}(B_{1}\times(-1,0])}\leq K\varepsilon^{-p}.
	\end{align}
From \eqref{AM20} and \eqref{AM21}, we see that no uniform spatial gradient estimate or temporal Lipschitz estimate on $B_{1}\times(-1,0]$ can depend only on $n,p$ and these bounded data, even with zero forcing and homogeneous lateral conditions.\ These are initial-layer counterexamples. So the construction does not refute estimates after a fixed positive waiting time.
	
	\emph{Idea of the construction.}\ The basic idea is to compress a spatial profile of fixed amplitude into the scale $\varepsilon$ and then let the corresponding bounded-domain $p$-Laplace flow evolve.\ Since the natural diffusion time scale is $\varepsilon^{p}$, the resulting sharp spatial concentration persists throughout an initial layer of this order, which directly destroys uniform spatial H\"{o}lder and gradient bounds.\ The same concentration is then converted, through the energy dissipation, into the blow-up of the temporal Lipschitz constant.
	
\end{remark}

%\begin{remark}
%Since our results cover the case of $p=2$, it doesn't need to analyze the stability for $p$ near $2$ any more.
%
%\end{remark}

%Especially our proof holds the promise of a wide application to more types of quasilinear parabolic equations.

The rest of the paper is organized as follows.\ In Section \ref{SEC002}, we present some preliminary results. The aim of Section \ref{SEC005} is to establish the forward-in-time propagation estimate for the spatial distribution, which is achieved by deriving a time-Lipschitz estimate after a fixed positive waiting time.\ In Section \ref{SEC006}, we derive the remaining two components of the parabolic De Giorgi truncation method, including the decay estimate of the space-time distribution and the measure-to-point oscillation reduction estimate, which together complete the proof of Theorem \ref{ZWTHM90}.\ Finally, Section \ref{SEC007} is dedicated to the establishment of space-time $W^{1,\infty}_{x,t}$ estimates after any given positive elapsed time.

\section{Preliminary}\label{SEC002}
We first fix some notation to be used throughout the paper. Given $x_{0}\in\mathbb{R}^{n}$ and $\rho>0$, $B_{\rho}(x_{0})$ denotes the ball centered at $x_{0}$ with radius $\rho$. For $t_{0}\leq0$ and $\tau>0$, we denote the backward parabolic cylinder by
\begin{align*}
[(x_{0},t_{0})+Q(\rho,\tau)]=B_{\rho}(x_{0})\times(t_{0}-\tau,t_{0}].
\end{align*}
To simplify the notation, we let $B_{\rho}=B_{\rho}(0)$ and $Q(\rho,\tau)=[(0,0)+Q(\rho,\tau)]$. 

We now list the classical isoperimetric inequality and the Sobolev embedding theorem.\ These two results constitute the essential tools required to set up the De Giorgi truncation method.
\begin{lemma}[see Lemma 2.2 in Chapter 2 of \cite{DGV2012}]
Let $n\geq2$ and $1<p<\infty$. For every $R>0$, $l>k$ and $u\in W^{1,1}(B_{R})$, we have
\begin{align}\label{pro001}
&(l-k)|\{u\geq l\}\cap B_{R}|\leq C(n,p)\frac{R^{n+1}}{|\{u\leq k\}\cap B_{R}|}\int_{\{k<u<l\}\cap B_{R}}|\nabla u|dx.
\end{align}
\end{lemma}
A direct consequence of Proposition 3.1 in Chapter I of \cite{D1993} gives the following parabolic Sobolev inequality.
\begin{lemma}\label{lem001}
Fix $n,p\geq2$.\ For any $B_{R}(x_{0})\subset\mathbb{R}^{n}$ and $-1\leq t_{1}<t_{2}\leq0$, let $u\in C((t_{1},t_{2});L^{p}(B_{R}(x_{0})))\cap L^{p}((t_{1},t_{2});W_{0}^{1,p}(B_{R}(x_{0})))$.\ Then we obtain
 \begin{align*}
	\|u\|_{L^{p\chi}(B_{R}(x_{0})\times(t_{1},t_{2}))}\leq C\bigg(\sup\limits_{t\in(t_{1},t_{2})}\int_{B_{R}(x_{0})}|u|^{p}dx+\int^{t_{2}}_{t_{1}}\int_{B_{R}(x_{0})}|\nabla u|^{p}dxdt\bigg)^{\frac{1}{p}},
\end{align*}
where $C=C(n,p)$ and $\chi=\frac{n+p}{n}.$

\end{lemma}

For $k\in\mathbb{R}$ and $u\in C([-1,0];L^{2}(B_{1}))\cap L^{p}((-1,0);W^{1,p}(B_{1}))$, write
\begin{align*}
	(u-k)_{+}=\max\{u-k,0\},\quad(u-k)_{-}=\max\{k-u,0\}.
\end{align*}
Proceeding as in the proof of Lemma 3.1 in \cite{HLZ2026}, with only minor modifications, we derive the following energy inequality.
\begin{lemma}\label{lem003}
Let $u$ solve \eqref{PO001} in $B_{1}\times(-1,0]$.\ For $k\in\mathbb{R}$, set $v_{\pm}=(u-k)_{\pm}$, and suppose that $[(x_{0},t_{0})+Q(\rho,\tau)]\subset B_{1}\times(-1,0]$.\ Then, for every $\xi\in C^{\infty}([(x_{0},t_{0})+Q(\rho,\tau)])$ such that $\xi=0$ on $\partial B_{\rho}(x_{0})\times(t_{0}-\tau,t_{0})$,  $0\leq\xi\leq1$, the following estimate holds for every $s\in(t_{0}-\tau,t_{0})$:
	\begin{align}\label{ENE01}
		&\int_{B_{\rho}(x_{0})}v_{\pm}^{2}\xi^{p}(x,s)dx+\frac{1}{6}\int_{B_{\rho}(x_{0})\times(t_{0}-\tau,s)}|\nabla(v_{\pm}\xi)|^{p}dxdt\notag\\
		&\leq\int_{B_{\rho}(x_{0})}v_{\pm}^{2}\xi^{p}(x,t_{0}-\tau)dx+C\int_{[(x_{0},t_{0})+Q(\rho,\tau)]}\big(v_{\pm}^{2}|\partial_{t}\xi|+v_{\pm}^{p}|\nabla\xi|^{p}\big)dxdt\notag\\
		&\quad+C\|f\|^{\frac{p}{p-1}}_{L^{\infty}(B_{1})}|[(x_{0},t_{0})+Q(\rho,\tau)]\cap \{v_{\pm}>0\}|.
	\end{align}
	
\end{lemma}

Let $r,p,q>1$ and $g\in C((-1,0];L^{p}(B_{1}))\cap L^{r}((-1,0);L^{q}(B_{1})).$ For $0<h\ll1$, the Steklov averages of $g$ are defined by 
\begin{align*}
	g_{h}(x,t)=
	\begin{cases}
		\frac{1}{h}\int^{t}_{t-h}g(x,s)ds,& t\in(-1+h,0),\\
		0,&t\leq-1+h,
	\end{cases}
\end{align*}
and
\begin{align*}
	g_{\bar{h}}(x,t)=
	\begin{cases}
		\frac{1}{h}\int^{t+h}_{t}g(x,s)ds,& t\in(-1,-h),\\
		0,&t\geq-h.
	\end{cases}
\end{align*}
\begin{lemma}\label{LEMZ01}
For $p>1$, suppose that $u$ is a weak solution of \eqref{PO001} in $Q(1,1)$. Then we have
	\begin{align}\label{DK01}
		\|u\|_{L^{\infty}(Q(1,1))}\leq\|\phi\|_{L^{\infty}(B_{1})}+\|f\|_{L^{\infty}(B_{1})}.
	\end{align}
\end{lemma}
\begin{remark}\label{RE09}
The global boundedness in Lemma \ref{LEMZ01} is established by a different approach from that used in Theorem 3.2 of \cite{HLZ2026}. More precisely, instead of employing a De Giorgi level-set iteration based on the energy inequality and the parabolic Sobolev inequality, we use a barrier argument together with the comparison principle in the following.	In addition, the estimate in \eqref{DK01} continues to hold under the Dirichlet condition of \eqref{AQ26}.
\end{remark}	

\begin{proof}
For $t\in(-1,0]$, define the spatially constant barrier as follows:
	\begin{align*}
		\bar{u}(t):=\|\phi\|_{L^{\infty}(B_{1})}+(t+1)\|f\|_{L^{\infty}(B_{1})}.
	\end{align*}
Note that $v(x,t):=u(x,t)-\bar{u}(t)$ solves
	\begin{align}\label{NEU07}
	\begin{cases}	
		\partial_{t}v-\mathrm{div}(|\nabla v|^{p-2}\nabla v)=f(x)-\|f\|_{L^{\infty}(B_{1})}\leq0,&\mathrm{in}\;Q(1,1),\\
		|\nabla v|^{p-2}\nabla v\cdot\nu=0,&\mathrm{on}\;\partial B_{1}\times(-1,0],\\
		v(x,-1)=\phi(x)-\|\phi\|_{L^{\infty}(B_{1})}\leq0,&\mathrm{in}\;B_{1},
	\end{cases}	
	\end{align}
where $\nu$ stands for the outward unit normal vector on $\partial B_{1}$.\ For $0<h\ll1$, set $\varphi=(v_{h})_{+}=\max\{v_{h},0\}$, where $v_{h}$ is the Steklov average of $v$.\ Taking $\varphi_{\bar{h}}\in W^{1,2}((-1,0);L^{2}(B_{1}))\cap L^{p}((-1,0);W^{1,p}(B_{1}))$ as a test function in \eqref{NEU07}, we obtain that for any $t\in(-1,0),$
\begin{align*}
	&\int_{B_{1}}v\varphi_{\bar{h}}dx\Big|_{-1}^{t}+\int_{-1}^{t}\int_{B_{1}}\big(-v\partial_{\tau}\varphi_{\bar{h}}+|\nabla v|^{p-2}\nabla v\cdot\nabla\varphi_{\bar{h}}\big)dxds\notag\\
	&=\int_{-1}^{t}\int_{B_{1}}\big(f(x)-\|f\|_{L^{\infty}(B_{1})}\big)\varphi_{\bar{h}}dxds.
\end{align*}	
This implies that
\begin{align}\label{AQ30}
	&\int_{-1}^{t}\int_{B_{1}}(\partial_{t}v_{h}\varphi+(|\nabla v|^{p-2}\nabla v)_{h}\cdot\nabla\varphi)dxds\notag\\
	&=\int_{-1}^{t}\int_{B_{1}}(f(x)-\|f\|_{L^{\infty}(B_{1})})\varphi dxds\leq0.
\end{align}	
By integrating by parts and using Lemma 3.2 in Chapter I of \cite{D1993}, we obtain 
\begin{align*}
	&\int_{-1}^{t}\int_{B_{1}}\partial_{t}v_{h}\varphi dxds=\frac{1}{2}\int_{-1}^{t}\int_{B_{1}}\partial_{t}[(v_{h})_{+}]^{2} dxds\notag\\
	&=\frac{1}{2}\int_{B_{1}}[(v_{h})_{+}]^{2}(x,t) dx-\frac{1}{2}\int_{B_{1}}[(v_{h})_{+}]^{2}(x,-1) dx\notag\\
	&\rightarrow\frac{1}{2}\int_{B_{1}}(v_{+})^{2}(x,t) dx-\frac{1}{2}\int_{B_{1}}(v_{+})^{2}(x,-1) dx,\quad\text{as }h\rightarrow0,
\end{align*}	
and
\begin{align*}
	&\int_{-1}^{t}\int_{B_{1}}(|\nabla v|^{p-2}\nabla v)_{h}\cdot\nabla\varphi dxds\rightarrow\int_{-1}^{t}\int_{B_{1}}|\nabla v_{+}|^{p}dxds,\quad\text{as }h\rightarrow0.
\end{align*}
Inserting the above terms into \eqref{AQ30}, it follows from \eqref{NEU07} that
\begin{align*}
	&\int_{B_{1}}(v_{+})^{2}(x,t) dx+\int_{-1}^{t}\int_{B_{1}}|\nabla v_{+}|^{p}dxds\leq \int_{B_{1}}(v_{+})^{2}(x,-1) dx=0,
\end{align*}	
which reads that for any $t\in(-1,0),$
\begin{align}\label{AQ89}
	u(x,t)\leq &\|\phi\|_{L^{\infty}(B_{1})}+(t+1)\|f\|_{L^{\infty}(B_{1})}\notag\\
	\leq&\|\phi\|_{L^{\infty}(B_{1})}+\|f\|_{L^{\infty}(B_{1})},\quad\text{for a.e.}\;x\in B_{1}.	
\end{align}		

On the other hand, observe that $w(x,t):=-u(x,t)$ satisfies
\begin{align*}
	\begin{cases}	
		\partial_{t}w-\mathrm{div}(|\nabla w|^{p-2}\nabla w)=-f(x)\leq0,&\mathrm{in}\;Q(1,1),\\
		|\nabla w|^{p-2}\nabla w\cdot\nu=0,&\mathrm{on}\;\partial B_{1}\times(-1,0],\\
		w(x,-1)=-\phi(x)\leq0,&\mathrm{in}\;B_{1}.
	\end{cases}	
\end{align*}	
By the same way, we find that for any $t\in(-1,0),$
\begin{align*}
	u(x,t)\geq0,\quad\text{for a.e.}\;x\in B_{1},
\end{align*}
which, combined with \eqref{AQ89}, shows that \eqref{DK01} holds.

\end{proof}

We conclude this section with the following inequality for the $p$-Laplace vector field, which reflects its strong monotonicity.
\begin{lemma}[see Chapter 12 in \cite{L2019}]\label{Lem0006}
	Let $n,p\geq2$. For every $a,b\in\mathbb{R}^{n}$, we obtain
	\begin{align*}
		\big(|a|^{p-2}a-|b|^{p-2}b\big)\cdot(a-b)\geq2^{2-p}|a-b|^{p}.
	\end{align*}

\end{lemma}

\section{Forward-in-time propagation of spatial distribution}\label{SEC005}

\subsection{Lipschitz continuity in time}

For simplicity, define
\begin{align}\label{DK02}
\mathcal{Q}=2\big(\|\phi\|_{L^{\infty}(B_{1})}+\|f\|_{L^{\infty}(B_{1})}\big).
\end{align}	
We now state the time-Lipschitz estimate required for the subsequent argument.
\begin{theorem}\label{thm10}
Set $p>2$. Let $u$ be the solution to problem \eqref{PO001} with $\Omega\times(-T,0]=B_{1}\times(-1,0]$. Then, for any $-1<s<t\leq0$,
	\begin{align}\label{AH01}
			\|u(\cdot,t)-u(\cdot,s)\|_{L^{\infty}(B_{1})}\leq\bigg(\frac{\mathcal{Q}}{(p-2)(s+1)}+\|f\|_{L^{\infty}(B_{1})}\bigg)(t-s),
		\end{align}	
and thus,
	\begin{align}\label{AH03}
		|\partial_{t}u(x,t)|\leq \frac{4\mathcal{Q}}{p-2}+\|f\|_{L^{\infty}(B_{1})},\quad\text{for a.e. }(x,t)\in B_{1}\times\Big(-\frac{3}{4},0\Big],
	\end{align}
where $\mathcal{Q}$ is given in \eqref{DK02}.
\end{theorem}
\begin{remark}\label{RE10}
A slight modification of the proof of Theorem \ref{thm10} shows that the conclusions in \eqref{AH01}--\eqref{AH03} also remain valid in the Dirichlet case of \eqref{AQ26}. However, when the homogeneous Neumann boundary condition is replaced by a bounded inhomogeneous one, a careful inspection reveals that the time-Lipschitz estimate in Theorem 3.1 cannot be derived within the present framework.
 
\end{remark}
\begin{proof}
For any given $-1<s<t\leq0,$ define the scaling factor $\lambda:=(\frac{t+1}{s+1})^{\frac{1}{p-2}}>1$. For $\tau\in(-1,s]$, write
\begin{align*}
\mathcal{M}_{\tau}=(\lambda-1)\|\phi\|_{L^{\infty}(B_{1})}+(1+\tau)(\lambda^{p-1}-1)\|f\|_{L^{\infty}(B_{1})},
\end{align*}
and
\begin{align*}
\hat{u}(x,\tau)=\lambda u(x,-1+\lambda^{p-2}(1+\tau)).	
\end{align*}	
Note that $\hat{u}$ solves 
\begin{align*}
	\begin{cases}
		\partial_{\tau}\hat{u}-\mathrm{div}(|\nabla \hat{u}|^{p-2}\nabla \hat{u})=\lambda^{p-1}f(x),& \mathrm{in}\;B_{1}\times(-1,s],\\
		|\nabla \hat{u}|^{p-2}\nabla \hat{u}\cdot\nu=0,&\mathrm{on}\;\partial B_{1}\times(-1,s],\\
		\hat{u}(x,-1)=\lambda\phi(x),&\mathrm{in}\; B_{1},
	\end{cases}
\end{align*}
where $\nu$ is the unit normal vector pointing outward on $\partial B_{1}$.\ Therefore, $v(x,\tau):=\hat{u}(x,\tau)-u(x,\tau)-\mathcal{M}_{\tau}\in C([-1,s];L^{2}(B_{1}))\cap L^{p}((-1,s);W^{1,p}(B_{1}))$ is the solution of
 \begin{align}\label{AH06}
 		&\partial_{\tau}v-\mathrm{div}(|\nabla \hat{u}|^{p-2}\nabla \hat{u}-|\nabla u|^{p-2}\nabla u)\notag\\
 		&=(\lambda^{p-1}-1)(f(x)-\|f\|_{L^{\infty}(B_{1})})\leq0,\quad \mathrm{in}\;B_{1}\times(-1,s],
 \end{align}	
subject to the boundary condition of
\begin{align}\label{AH1001}
(|\nabla \hat{u}|^{p-2}\nabla \hat{u}-|\nabla u|^{p-2}\nabla u)\cdot\nu=0,\quad\mathrm{on}\;\partial B_{1}\times(-1,s],
\end{align}	
and
\begin{align}\label{AH10}
v(x,-1)=&(\lambda-1)(\phi(x)-\|\phi\|_{L^{\infty}(B_{1})})\leq0,\quad\mathrm{in}\;B_{1}.	
\end{align}		
For $0<h\ll 1$, define $\varphi=(v_{h})_{+}=\max\{v_{h},0\}$, where $v_{h}$ denotes the Steklov average of $v$. Testing \eqref{AH06} with $\varphi_{\bar{h}}\in W^{1,2}((-1,s);L^{2}(B_{1}))\cap L^{p}((-1,s);W^{1,p}(B_{1}))$, we have from \eqref{AH1001} that
\begin{align*}
	&\int_{B_{1}}v\varphi_{\bar{h}}dx\Big|_{-1}^{s}+\int_{-1}^{s}\int_{B_{1}}\big(-v\partial_{\tau}\varphi_{\bar{h}}+(|\nabla \hat{u}|^{p-2}\nabla \hat{u}-|\nabla u|^{p-2}\nabla u)\cdot\nabla\varphi_{\bar{h}}\big)dxd\tau\notag\\
	&=\int_{-1}^{s}\int_{B_{1}}(\lambda^{p-1}-1)(f(x)-\|f\|_{L^{\infty}(B_{1})})\varphi_{\bar{h}}dxd\tau,
\end{align*}	
which leads to that
\begin{align}\label{AH08}
	&\int_{-1}^{s}\int_{B_{1}}(\partial_{\tau}v_{h}\varphi+(|\nabla \hat{u}|^{p-2}\nabla \hat{u}-|\nabla u|^{p-2}\nabla u)_{h}\cdot\nabla\varphi)dxd\tau\notag\\
	&=\int_{-1}^{s}\int_{B_{1}}(\lambda^{p-1}-1)(f(x)-\|f\|_{L^{\infty}(B_{1})})\varphi dxd\tau\leq0.
\end{align}	
Using integration by parts together with Lemma 3.2 in Chapter I of \cite{D1993}, we deduce
\begin{align*}
&\int_{-1}^{s}\int_{B_{1}}\partial_{\tau}v_{h}\varphi dxd\tau=\frac{1}{2}\int_{-1}^{s}\int_{B_{1}}\partial_{\tau}[(v_{h})_{+}]^{2} dxd\tau\notag\\
&=\frac{1}{2}\int_{B_{1}}[(v_{h})_{+}]^{2}(x,s) dx-\frac{1}{2}\int_{B_{1}}[(v_{h})_{+}]^{2}(x,-1) dx\notag\\
&\rightarrow\frac{1}{2}\int_{B_{1}}(v_{+})^{2}(x,s) dx-\frac{1}{2}\int_{B_{1}}(v_{+})^{2}(x,-1) dx,\quad\text{as }h\rightarrow0,
\end{align*}	
and
\begin{align*}
&\int_{-1}^{s}\int_{B_{1}}(|\nabla \hat{u}|^{p-2}\nabla \hat{u}-|\nabla u|^{p-2}\nabla u)_{h}\cdot\nabla\varphi dxd\tau\notag\\
&\rightarrow\int_{(B_{1}\times(-1,s))\cap \{v_{+}>0\}}(|\nabla \hat{u}|^{p-2}\nabla \hat{u}-|\nabla u|^{p-2}\nabla u)\cdot\nabla (\hat{u}-u)dxd\tau,\quad\text{as }h\rightarrow0.
\end{align*}
Making use of Lemma \ref{Lem0006}, we have
\begin{align*}
\int_{(B_{1}\times(-1,s))\cap \{v_{+}>0\}}(|\nabla \hat{u}|^{p-2}\nabla \hat{u}-|\nabla u|^{p-2}\nabla u)\cdot\nabla (\hat{u}-u)dxd\tau\geq0.	
\end{align*}	
Inserting the above terms into \eqref{AH08}, we infer from \eqref{AH10} that
\begin{align*}
&\int_{B_{1}}(v_{+})^{2}(x,s) dx\leq \int_{B_{1}}(v_{+})^{2}(x,-1) dx=0.
\end{align*}	
This implies that 
\begin{align}\label{AH13}
\hat{u}(x,s)-u(x,s)\leq \mathcal{M}_{\tau},\quad\text{for a.e.}\;x\in B_{1}.	
\end{align}	

On the other hand, the difference $w:=u-\hat{u}$ satisfies
 \begin{align*}
 \begin{cases}	
\partial_{\tau}w-\mathrm{div}(|\nabla u|^{p-2}\nabla u-|\nabla \hat{u}|^{p-2}\nabla \hat{u})=(1-\lambda^{p-1})f(x)\leq0,\;\, \mathrm{in}\;B_{1}\times(-1,s],\\
(|\nabla u|^{p-2}\nabla u-|\nabla \hat{u}|^{p-2}\nabla \hat{u})\cdot\nu=0,\;\, \mathrm{on}\;\partial B_{1}\times(-1,s],\\
w(x,-1)=(1-\lambda)\phi(x)\leq0,\;\,\mathrm{in}\;B_{1}.
\end{cases}
\end{align*}	
Arguing as above, we obtain
\begin{align*}
u(x,s)-\hat{u}(x,s)\leq 0,\quad\text{for a.e.}\;x\in B_{1},
\end{align*}
 which, in combination with \eqref{AH13}, reads that for $\tau\in(-1,s]$,
\begin{align*}
\|u(\cdot,s)-\hat{u}(\cdot,s)\|_{L^{\infty}(B_{1})}\leq \mathcal{M}_{\tau}\leq(\lambda-1)\|\phi\|_{L^{\infty}(B_{1})}+(1+s)(\lambda^{p-1}-1)\|f\|_{L^{\infty}(B_{1})}.
\end{align*}	
Recalling the definition of $\lambda$, we obtain
\begin{align*}
\tilde{u}(x,s)=\lambda u(x,t),\quad \text{for }x\in B_{1}\text{ and }-1<s<t\leq0.	
\end{align*}	
It then follows that for $-1<s<t\leq0$,
\begin{align}\label{AQ12}
&\|u(\cdot,t)-u(\cdot,s)\|_{L^{\infty}(B_{1})}\leq(1-\lambda^{-1})\|u(\cdot,s)\|_{L^{\infty}(B_{1})}+\|u(\cdot,t)-\lambda^{-1}u(\cdot,s)\|_{L^{\infty}(B_{1})}\notag\\
&\leq(1-\lambda^{-1})(\|u\|_{L^{\infty}(Q(1,1))}+\|\phi\|_{L^{\infty}(B_{1})})+(s+1)(\lambda^{p-2}-\lambda^{-1})\|f\|_{L^{\infty}(B_{1})}.
\end{align}	
Observe from the definition of $\lambda$ that for $p>2$ and $-1<s<t\leq0$,
\begin{align}\label{AQ15}
1-\lambda^{-1}=&1-\Big(\frac{t+1}{s+1}\Big)^{-\frac{1}{p-2}}=1-e^{-\frac{1}{p-2}\ln\frac{t+1}{s+1}}\notag\\
\leq&\frac{1}{p-2}\ln\frac{t+1}{s+1}\leq\frac{t-s}{(p-2)(s+1)},
\end{align}	
and
\begin{align}\label{AQ16}
&(s+1)(\lambda^{p-2}-\lambda^{-1})=t+1-\lambda^{-1}(s+1)\notag\\
&=t-s+(1-\lambda^{-1})(s+1)\leq t-s+1-\lambda^{-1},
\end{align}	
where, in deriving \eqref{AQ15}, we also used the following inequalities:
\begin{align*}
1-e^{-k\beta}\leq k\beta,\quad\ln(1+\beta)\leq\beta,\quad \text{for any }k,\beta\geq0.
\end{align*}	
Substituting \eqref{AQ15}--\eqref{AQ16} into \eqref{AQ12}, we obtain from Lemma \ref{LEMZ01} that for $-1<s<t\leq0$,
\begin{align*}
	\|u(\cdot,t)-u(\cdot,s)\|_{L^{\infty}(B_{1})}\leq&\mathcal{Q}(1-\lambda^{-1})+\|f\|_{L^{\infty}(B_{1})}(t-s)\notag\\
	\leq&\bigg(\frac{\mathcal{Q}}{(p-2)(s+1)}+\|f\|_{L^{\infty}(B_{1})}\bigg)(t-s),
\end{align*}	
where $\mathcal{Q}$ is introduced in \eqref{DK02}.

\end{proof}

\subsection{The measure expansion in time}
Note that for any $R\in(0,\frac{1}{2}]$, $x_{0}\in\overline{B}_{\frac{1}{2}}$ and $t_{0}\in[-\frac{1}{2},0]$, we have
\begin{align*}
	[(x_{0},t_{0})+Q(R,R^{2})]\subset B_{1}\times\Big(-\frac{3}{4},0\Big]\subset Q(1,1).
\end{align*}
Without loss of generality, after a suitable translation, we take $(x_0,t_0)=(0,0)$. Write
\begin{align*}
	\mu^{+}=\sup\limits_{Q(R,R^{2})}u,\quad\mu^{-}=\inf\limits_{Q(R,R^{2})}u,\quad \omega=\mathop{osc}\limits_{Q(R,R^{2})}u=\mu^{+}-\mu^{-}.
\end{align*}
Define the cylinder as follows:
\begin{align}\label{Z999}
	Q(R,a_{0}R^{p}),\quad a_{0}=\left(\frac{\sigma\omega}{A}\right)^{2-p},
\end{align}
where $\sigma\in(0,1)$ and $A\in(1,\infty)$ are constants depending only on the prescribed data, whose values will be determined in the following.\ It should be emphasized that $a_{0}$ plays the role of the intrinsic scaling factor, and the rescaled cylinders $Q(R,a_{0}R^{p})$ are tailored to the degenerate structure of the equation.\ The exponent $2-p$ involved in this scaling had already appeared in the earlier works \cite{D1986,D1993}.

There are two possibilities: $\omega\leq \sigma^{-1}AR$ or $\omega>\sigma^{-1}AR$. If the latter occurs, we have
\begin{align*}
	Q(R,a_{0}R^{p})\subset Q(R,R^{2}),\quad\mathop{osc}\limits_{Q(R,a_{0}R^{p})}u\leq\omega.
\end{align*}
Define
\begin{align}\label{M01}
	m=\Big(\frac{\sigma\omega}{\mathcal{M}_{\ast}}\Big)^{2-p},\quad\mathcal{M}_{\ast}=2\big(\|\phi\|_{L^{\infty}(B_{1})}+\|f\|_{L^{\infty}(B_{1})}\big)+1,
\end{align}
which, together with \eqref{DK01}, shows that $m\geq1$, and
\begin{align}\label{M02}
	\mathcal{M}_{\ast}\geq2\|u\|_{L^{\infty}(Q(1,1))}+1.
\end{align}	
If $A\geq\mathcal{M}_{\ast}$, then we have $Q(R,mR^{p})\subset Q(R,a_{0}R^{p})$. 

Define
\begin{align}\label{AH26}
	\kappa_{0}:=-\frac{\ln(1-\sqrt{15}/4)}{\ln 2}>0.	
\end{align}	
As a consequence of the time-Lipschitz estimate established in Theorem \ref{thm10}, we obtain the following forward-in-time propagation estimate for the spatial measure of level sets. 
\begin{prop}\label{Pro006}
	Let $\sigma$, $m$, and $\mathcal{M}_{\ast}$ be introduced in \eqref{Z999}--\eqref{M01}. Then one can choose a small constant $\sigma_{0}\in(0,1]$, with its value determined only by the prescribed data and independent of $\sigma$, so that for any $0<\sigma<\sigma_{0}$,
	
	$(i)$ if
	\begin{align}\label{WMDN001}
		\frac{|B_{R}\cap\{u(\cdot,-mR^{p})>\mu^{+}-2^{-1}\omega\}|}{|B_{R}|}\leq\frac{1}{2},
	\end{align}
	then we obtain either $\omega\leq \sigma^{-1}\mathcal{M}_{\ast}R$, or
	\begin{align}\label{DQAF001}
		\frac{|B_{R}\cap\{u(\cdot,t)>\mu^{+}-2^{-(\kappa_{0}+1)}\omega\}|}{|B_{R}|}\leq\frac{4}{5},\quad\text{for all }t\in[-mR^{p},0];
	\end{align}
	
	$(ii)$ if
	\begin{align*}
		\frac{|B_{R}\cap\{u(\cdot,-R^{p})<\mu^{-}+2^{-1}\omega\}|}{|B_{R}|}\leq\frac{1}{2},
	\end{align*}
	then we obtain either $\omega\leq \sigma^{-1}\mathcal{M}_{\ast}R$, or
	\begin{align}\label{WZPMQ001}
		\frac{|B_{R}\cap\{u(\cdot,t)<\mu^{-}+2^{-(\kappa_{0}+1)}\omega\}|}{|B_{R}|}\leq\frac{4}{5},\quad\text{for all }t\in[-mR^{p},0],
	\end{align}
	where $\kappa_{0}$ is defined in \eqref{AH26}.
	
\end{prop}

\begin{proof}
	It suffices to prove \eqref{DQAF001}, since \eqref{WZPMQ001} follows by the same argument. To simplify the notation, set $v_{+}:=(u-(\mu^{+}-2^{-1}\omega))_{+}$. Define
	\begin{align}\label{A09}
		\mathcal{G}:=\frac{4\mathcal{Q}}{p-2}+\|f\|_{L^{\infty}(B_{1})},\quad\sigma_{0}:=\min\Big\{1,\Big(\frac{\mathcal{M}_{\ast}}{8\mathcal{G}}\Big)^{\frac{1}{p-1}}\Big\},	
	\end{align}	
	where $\mathcal{Q}$ and $\mathcal{M}_{\ast}$ are, respectively, given by \eqref{DK02} and \eqref{M01}. From \eqref{AH03} and \eqref{WMDN001}, it follows that for any $t\in[-mR^{p},0]$ and $\sigma\in(0,\sigma_{0})$, if 
	$$\omega>\sigma^{-1}\mathcal{M}_{\ast}R\geq \big(8\mathcal{G}(\sigma^{-1}\mathcal{M}_{\ast})^{p-2}\big)^{\frac{1}{p-1}}R^{\frac{p}{p-1}},$$
	then
	\begin{align}\label{QNE089}
		\int_{B_{R}}v_{+}^{2}(x,t)dx=&\int_{B_{R}}v_{+}^{2}(x,-mR^{p})dx+\int_{B_{R}}(v_{+}^{2}(x,t)-v_{+}^{2}(x,-mR^{p}))dx\notag\\
		\leq&\frac{\omega^{2}}{8}|B_{R}|+\omega\int_{B_{R}}|v_{+}(x,t)-v_{+}(x,-mR^{p})|dx\notag\\
		\leq& \Big(\frac{\omega^{2}}{8}+\frac{\mathcal{G}\omega mR^{p}}{2}\Big)|B_{R}|\leq \frac{3\omega^{2}}{16}|B_{R}|,
	\end{align}
	where $\mathcal{G}$ is defined by \eqref{A09}. 
	
	On the other hand, for any $t\in[-mR^{p},0]$, we have 
	\begin{align*}
		&\int_{B_{R}}v_{+}^{2}(x,t)dx\geq\frac{\omega^{2}}{4}(1-2^{-\kappa_{0}})^{2}|B_{R}\cap\{u(\cdot,t)>\mu^{+}-2^{-(\kappa_{0}+1)}\omega\}|,
	\end{align*}
	where $\kappa_{0}$ is given by \eqref{AH26}.\ Substituting this into \eqref{QNE089}, we obtain that for all $t\in[-mR^{p},0]$, 
	\begin{align*}
		&\frac{|B_{R}\cap\{u(\cdot,t)>\mu^{+}-2^{-(\kappa_{0}+1)}\omega\}|}{|B_{R}|}\leq\frac{3}{4(1-2^{-\kappa_{0}})^{2}}\leq\frac{4}{5}.
	\end{align*}
	The proof is complete.

\end{proof}

\section{H\"{o}lder regularity of weak solutions }\label{SEC006}

With the time-expansion estimate already available, it remains to establish the decay estimate and the measure-to-point oscillation improvement estimate in order to implement the parabolic De Giorgi truncation method and thereby obtain the H\"{o}lder continuity of the solution.
\begin{lemma}\label{lem009}
Let $\sigma,\,m,\,\mathcal{M}_{\ast}$ be as in \eqref{Z999}--\eqref{M01}.\ Then one can find a constant $\widehat{C}_{0}>1$, depending only on the data but not on $\sigma$, such that for every $\gamma\in(0,1)$ and $\delta\in(0,\frac{1}{2}]$,
	
	$(i)$ if
	\begin{align}\label{AH20}
		\frac{| B_{R/2}\cap\{u(\cdot,t)>\mu^{+}-\delta\omega\}|}{|B_{R/2}|}\leq1-\gamma,\quad\forall\;t\in[-mR^{p},0],
	\end{align}
	then for any given $j\geq1$ and $0<\sigma<\min\{1,2^{-j}\delta\mathcal{M}_{\ast}\}$, we obtain either $\omega\leq\sigma^{-1}\mathcal{M}_{\ast}R$, or
	\begin{align}\label{AH21}
	\frac{| Q(R/2,m(R/2)^{p})\cap\{u>\mu^{+}-\frac{\delta\omega}{2^{j}}\}|}{|Q(R/2,m(R/2)^{p})|}\leq\frac{\widehat{C}_{0}}{\gamma j^{\frac{p-1}{p}}};
	\end{align}
	
	$(ii)$ if
	\begin{align}\label{AH22}
		\frac{|B_{R/2}\cap\{u(\cdot,t)<\mu^{-}+\delta\omega\}|}{|B_{R/2}|}\leq1-\gamma,\quad\forall\;t\in[-mR^{p},0],
	\end{align}
	then for any given $j\geq1$ and $0<\sigma<\min\{1,2^{-j}\delta\mathcal{M}_{\ast}\}$, we obtain either $\omega\leq\sigma^{-1}\mathcal{M}_{\ast}R$, or
	\begin{align}\label{AH23}
	\frac{| Q(R/2,m(R/2)^{p})\cap\{u<\mu^{-}+\frac{\delta\omega}{2^{j}}\}|}{|Q(R/2,m(R/2)^{p})|}\leq\frac{\widehat{C}_{0}}{\gamma j^{\frac{p-1}{p}}}.
	\end{align}
\end{lemma}

\begin{proof}
	\noindent{\bf Step 1.}
	For $i\geq0$, denote $k_{i}=\mu^{+}-\frac{\delta\omega}{2^{i}}$, and
\begin{align*}
	A_{i}(t)=B_{R/2}\cap\{u(\cdot,t)>k_{i}\},\quad A_{i}=Q(R/2,m(R/2)^{p})\cap\{u>k_{i}\}.
\end{align*}
From \eqref{AH20}, we obtain
	\begin{align*}
		|B_{R/2}\setminus A_{i}(t)|\geq\gamma|B_{R/2}|\geq C(n)\gamma R^{n}.
	\end{align*}
Combining this with \eqref{pro001}, we have
	\begin{align*}
		|A_{i+1}(t)|\leq\frac{C2^{i}R}{\gamma\delta\omega}\int_{A_{i}(t)\setminus A_{i+1}(t)}|\nabla u|dx,
	\end{align*}
which implies that
	\begin{align*}
	|A_{i+1}|\leq\frac{C2^{i}R}{\gamma\delta\omega}\int_{A_{i}\setminus A_{i+1}}|\nabla u|dxdt.
\end{align*}	
Since $p>2$, H\"{o}lder's inequality yields 
	\begin{align*}
	\int_{A_{i}\setminus A_{i+1}}|\nabla u|dxdt\leq&\bigg(\int_{A_{i}\setminus A_{i+1}}|\nabla u|^{p}dxdt\bigg)^{\frac{1}{p}}|A_{i}\setminus A_{i+1}|^{\frac{p-1}{p}}\notag\\
		\leq&\bigg(\int_{Q(R/2,m(R/2)^{p})}|\nabla (u-k_{i})_{+}|^{p}dxdt\bigg)^{\frac{1}{p}}|A_{i}\setminus A_{i+1}|^{\frac{p-1}{p}}.
	\end{align*}
Let $\zeta\in C_{0}^{\infty}(B_{R})$ be a time-independent cutoff function satisfying that
	\begin{align*}
		\zeta=1\;\mathrm{in}\;B_{R/2},\quad\mathrm{and}\;\,0\leq\zeta\leq1,\;|\nabla\zeta|\leq\frac{4}{R}\;\,\mathrm{in}\;B_{R}.
	\end{align*}
It follows from \eqref{ENE01} and \eqref{M01} that for $0<\sigma< \min\{1,2^{-i}\delta\mathcal{M}_{\ast}\}$, if 
\begin{align*}
\omega>\sigma^{-1}\mathcal{M}_{\ast}R\geq 2^{i}\delta^{-1}R,
\end{align*}
then
	\begin{align*}
		&\int_{Q(R/2,m(R/2)^{p})}|\nabla (u-k_{i})_{+}|^{p}dxdt\leq\int_{Q(R,m(R/2)^{p})}|\nabla ((u-k_{i})_{+}\zeta)|^{p}dxdt\notag\\
		&\leq 6\int_{B_{R}} (u-k_{i})^{2}_{+}(x,-m(R/2)^{p})\zeta^{p}(x)dx+C\int_{Q(R,m(R/2)^{p})}(u-k_{i})_{+}^{p}|\nabla\zeta|^{p}dxdt\notag\\
		&\quad+C|Q(R,m(R/2)^{p})\cap\{(u-k_{i})_{+}>0\}|\\
		&\leq C\Big(\frac{(\delta\omega)^{2}}{2^{2i}}+\frac{mR^{n}(\delta\omega)^{p}}{2^{pi}}+mR^{n+p}\Big)\leq \frac{CmR^{n}(\delta\omega)^{p}}{2^{pi}}.
	\end{align*}
Putting the above facts together, we deduce
	\begin{align*}
		|A_{i+1}|\leq&\frac{C}{\gamma}|A_{i}\setminus A_{i+1}|^{\frac{p-1}{p}}|Q(R/2,m(R/2)^{p})|^{\frac{1}{p}}.
	\end{align*}
Therefore, for any $j>i\geq0$ and $0<\sigma<\min\{1,2^{-j}\delta\mathcal{M}_{\ast}\}$, if $\omega>\sigma^{-1}\mathcal{M}_{\ast}R,$ then
	\begin{align*}
		j|A_{j}|^{\frac{p}{p-1}}\leq&\sum^{j-1}_{i=0}|A_{i+1}|^{\frac{p}{p-1}}\leq\frac{C}{\gamma^{\frac{p}{p-1}}}|Q(R/2,m(R/2)^{p})|^{\frac{p}{p-1}}.
	\end{align*}
This proves \eqref{AH21}.
	
\noindent{\bf Step 2.}
For $i\geq0$, let $\tilde{k}_{i}=\mu^{-}+\frac{\delta\omega}{2^{i}}$ and
	\begin{align*}
		\tilde{A}_{i}(t)=B_{mR/2}\cap\{u(\cdot,t)<k_{i}\},\quad \tilde{A}_{i}=Q(R/2,m(R/2)^{p})\cap\{u<k_{i}\}.
	\end{align*}
From \eqref{AH22}, we derive
	\begin{align*}
		|B_{R/2}\setminus \tilde{A}_{i}(t)|\geq\gamma|B_{R/2}|\geq C(n)\gamma R^{n}.
	\end{align*}
It then follows from \eqref{pro001} that
	\begin{align*}
	|\tilde{A}_{i+1}|\leq\frac{C2^{i}R}{\gamma\delta\omega}\int_{\tilde{A}_{i}\setminus \tilde{A}_{i+1}}|\nabla u|dxdt.
\end{align*}	
Proceeding as in the proof of \eqref{AH21}, we conclude that \eqref{AH23} holds.

\end{proof}

\begin{lemma}\label{LEM0035}
Let $\sigma,m,\mathcal{M}_{\ast}$ be as in \eqref{Z999}--\eqref{M01}.\ There exists a constant $\gamma_{0}\in(0,1)$, whose dependence is only on the prescribed data and not on $\sigma$, such that

$(i)$ if
\begin{align}\label{ZWZ007}
|Q(R,mR^{p})\cap\{u>\mu^{+}-\sigma\omega\}|\leq\gamma_{0}|Q(R,mR^{p})|,
\end{align}
then we derive either $\omega\leq \sigma^{-1}\mathcal{M}_{\ast}R$, or
\begin{align}\label{DZ001}
u\leq \mu^{+}-\frac{\sigma\omega}{2},\quad\mathrm{in}\text{ $Q(R/2,m(R/2)^{p})$;}
\end{align}

$(ii)$ if
\begin{align*}
|Q(R,mR^{p})\cap\{u<\mu^{-}+\sigma\omega\}|\leq\gamma_{0}|Q(R,mR^{p})|,
\end{align*}
then we derive either $\omega\leq \sigma^{-1}\mathcal{M}_{\ast}R$, or
\begin{align}\label{ZWZ009}
u\geq \mu^{-}+\frac{\sigma\omega}{2},\quad\mathrm{in}\text{ $Q(R/2,m(R/2)^{p})$.}
\end{align}

\end{lemma}

\begin{proof}
We only need to prove \eqref{DZ001}, as \eqref{ZWZ009} can be established by the same argument. For $i\geq0$, let
\begin{align*}
r_{i}=\frac{R}{2}+\frac{R}{2^{i+1}},\quad k_{i}=\mu^{+}-\sigma\omega+\frac{\sigma\omega}{2}(1-2^{-i}).
\end{align*}
Let $\xi_{i}\in C^{\infty}(Q(r_{i},mr_{i}^{p}))$ denote a smooth cutoff function satisfying the following properties:\ $0\leq\xi_{i}\leq1$ in $Q(r_{i},mr_{i}^{p})$, $\xi_{i}=1$ in $Q(r_{i+1},mr_{i+1}^{p})$, $\xi_{i}=0$ on $\partial_{pa}Q(r_{i},mr_{i}^{p}),$ and
\begin{align*}
|\nabla\xi_{i}|\leq\frac{2^{i+3}}{R},\quad|\partial_{t}\xi_{i}|\leq\frac{2}{m(r_{i}^{p}-r_{i+1}^{p})}\leq\frac{C2^{i}}{mR^{p}},\quad \mathrm{in}\;Q(r_{i},mr_{i}^{p}).
\end{align*}
In light of $p>2$, we have from Lemma \ref{lem003} and \eqref{M01} that if $\omega>\sigma^{-1}\mathcal{M}_{\ast}R$,
\begin{align}\label{WQ932}
	&\sup\limits_{t\in(-mr_{i}^{p},0)}(\sigma\omega)^{2-p}\int_{B_{r_{i}}}(u-k_{i})_{+}^{p}\xi_{i}^{p}dx+\frac{1}{6}\int_{Q(r_{i},mr_{i}^{p})}|\nabla((u-k_{i})_{+}\xi_{i})|^{p}dxdt\notag\\
	&\leq\sup\limits_{t\in(-mr_{i}^{p},0)}\int_{B_{r_{i}}}(u-k_{i})_{+}^{2}\xi_{i}^{p}dx+\frac{1}{6}\int_{Q(r_{i},mr_{i}^{p})}|\nabla((u-k_{i})_{+}\xi_{i})|^{p}dxdt\notag\\
	&\leq C\int_{Q(r_{i},mr_{i}^{p})}\big((u-k_{i})_{+}^{2}|\partial_{t}\xi_{i}|+(u-k_{i})_{+}^{p}|\nabla\xi_{i}|^{p}\big)dxdt\notag\\
	&\quad+C\|f\|_{L^{\infty}(B_{1})}|Q(r_{i},mr_{i}^{p})\cap\{u>k_{i}\}|\notag\\
	&\leq C\Big(\frac{2^{pi}(\sigma\omega)^{p}}{R^{p}}+1\Big)|Q(r_{i},mr_{i}^{p})\cap\{u>k_{i}\}|\notag\\
	&\leq \frac{C2^{pi}(\sigma\omega)^{p}}{R^{p}}|Q(r_{i},mr_{i}^{p})\cap\{u>k_{i}\}|.
\end{align}
Denote $\tilde{u}(x,\tilde{t})=u(x,m\tilde{t})$, $\tilde{\xi}_{i}(x,\tilde{t})=\xi_{i}(x,m\tilde{t})$, and
\begin{align*}
	\tilde{A}_{i}(\tilde{t})=B_{r_{i}}\cap\{\tilde{u}(\cdot,\tilde{t})>k_{i}\},\quad\tilde{A}_{i}=Q(r_{i},r_{i}^{p})\cap\{\tilde{u}>k_{i}\}.
\end{align*}
By Lemma \ref{lem001} and \eqref{WQ932}, we have
\begin{align*}
	&\frac{(\sigma\omega)^{p}}{2^{p(i+2)}}|\tilde{A}_{i+1}|^{\frac{1}{\chi}}=(k_{i+1}-k_{i})^{p}|\tilde{A}_{i+1}|^{\frac{1}{\chi}}\leq\|(\tilde{u}_{i}-k_{i})_{+}\tilde{\xi}_{i}\|^{p}_{L^{p\chi}(Q(r_{i},r_{i}^{p}))}\notag\\
	&\leq\sup\limits_{\tilde{t}\in(-r_{i}^{p},0)}\int_{B_{r_{i}}}(\tilde{u}-k_{i})_{+}^{p}\tilde{\xi}_{i}^{p}dx+\int_{Q(r_{i},r_{i}^{p})}|\nabla((\tilde{u}-k_{i})_{+}\tilde{\xi}_{i})|^{p}dxd\tilde{t}\notag\\
	&\leq \frac{C2^{pi}(\sigma\omega)^{p}}{R^{p}}|Q(r_{i},r_{i}^{p})\cap\{\tilde{u}>k_{i}\}|,
\end{align*}
where $\chi=\frac{n+p}{n}$. Hence, if $\omega>\sigma^{-1}\mathcal{M}_{\ast}R$, we obtain
\begin{align*}
	|\tilde{A}_{i+1}|\leq&\Big(\frac{C4^{pi}}{R^{p}}|\tilde{A}_{i}|\Big)^{\chi}.
\end{align*}
Write
\begin{align*}
	F_{i}:=\frac{|\tilde{A}_{i}|}{|Q(R,R^{p})|}.
\end{align*}
Then we have
\begin{align*}
	F_{i+1}\leq&(C4^{pi})^{\chi}F_{i}^{\chi}\leq\prod\limits^{i}_{s=0}(C4^{p(i-s)})^{\chi^{s+1}}F_{0}^{\chi^{i+1}}\leq(C_{0}F_{0})^{\chi^{i+1}},
\end{align*}
where $C_{0}>0$ depends only on the prescribed data and is independent of $\sigma$. Then picking $\gamma_{0}=(2C_{0})^{-1}$ in \eqref{ZWZ007}, we derive
\begin{align*}
F_{i+1}\leq (C_{0}\gamma_{0})^{\chi^{i+1}}\leq 2^{-\chi^{i+1}}\rightarrow0,\quad\text{as }i\rightarrow\infty.
\end{align*}
This completes the proof.

\end{proof}

Combining Proposition \ref{Pro006}, Lemmas \ref{lem009} and \ref{LEM0035}, we are now ready to establish the following improvement on oscillation of the solution.
\begin{corollary}\label{AMZW01}
The constants $\sigma$ and $A$ can be determined first, and then $a_{0}$ and $m$ can be fixed, together with a constant $\kappa_{\ast}>1$, all depending solely on the prescribed data, such that

$(i)$ if
\begin{align}\label{W008}
\frac{|B_{R/2}\cap\{u(\cdot,-m(R/2)^{p})>\mu^{+}-2^{-1}\omega\}|}{|B_{R/2}|}\leq\frac{1}{2},
\end{align}
then we derive either $\omega\leq \sigma^{-1}AR=2^{\kappa_{\ast}}\mathcal{M}_{\ast}R$, or
\begin{align}\label{W009}
\sup\limits_{Q(R/4,m(R/4)^{p})}u\leq \mu^{+}-\frac{\omega}{2^{\kappa_{\ast}+1}};
\end{align}

$(ii)$ if
\begin{align*}
\frac{|B_{R/2}\cap\{u(\cdot,-m(R/2)^{p})<\mu^{-}+2^{-1}\omega\}|}{|B_{R/2}|}\leq\frac{1}{2},
\end{align*}
then we derive either $\omega\leq \sigma^{-1}AR=2^{\kappa_{\ast}}\mathcal{M}_{\ast}R$, or
\begin{align}\label{W010}
\inf\limits_{Q(R/4,m(R/4)^{p})}u\geq \mu^{-}+\frac{\omega}{2^{\kappa_{\ast}+1}}.
\end{align}

\end{corollary}
\begin{proof}
Since \eqref{W010} follows by an analogous argument, it is enough to prove \eqref{W009}. We first fix a sufficiently large constant $j_{0}>3$, depending only on the prescribed data, such that
\begin{align}\label{W012}
2^{-(\kappa_{0}+1+j_{0})}\leq\sigma_{0},\quad5\widehat{C}_{0}j_{0}^{-\frac{p-1}{p}}\leq\gamma_{0},
\end{align}
where $\sigma_{0},\,\kappa_{0},\,\widehat{C}_{0},\,\gamma_{0}$ are, respectively, provided by Proposition \ref{Pro006}, \eqref{AH26}, Lemmas \ref{lem009} and \ref{LEM0035}.\ Therefore, applying Proposition \ref{Pro006} with \eqref{W008} and utilizing Lemma \ref{lem009} with $\delta=\frac{1}{2^{\kappa_{0}+1}}$, $\sigma=\frac{1}{2^{\kappa_{0}+1+j_{0}}}$, $\gamma=\frac{1}{5}$, and $j=j_{0}$, it follows from \eqref{W012} that either $\omega\leq \sigma^{-1}AR=2^{\kappa_{0}+1+j_{0}}AR$ with $A=\mathcal{M}_{\ast}$,
or
\begin{align}\label{W013}
&\frac{|Q(R/2,m(R/2)^{p})\cap\{u>\mu^{+}-\frac{\omega}{2^{\kappa_{0}+1+j_{0}}}\}|}{|Q(R/2,m(R/2)^{p})|}\leq5\widehat{C}_{0}j_{0}^{-\frac{p-1}{p}}\leq\gamma_{0}.
\end{align}
Using Lemma \ref{LEM0035} with $\sigma=\frac{1}{2^{\kappa_{0}+1+j_{0}}}$, we deduce from \eqref{W013} that either $\omega\leq \sigma^{-1}AR=2^{\kappa_{0}+1+j_{0}}\mathcal{M}_{\ast}R$, or
\begin{align*}
u\leq \mu^{+}-\frac{\omega}{2^{\kappa_{0}+2+j_{0}}},\quad\mathrm{in}\text{ $Q(R/4,m(R/4)^{p})$.}
\end{align*}
Then by fixing $\kappa_{\ast}:=\kappa_{0}+1+j_{0}$, we complete the proof.

\end{proof}

\subsection{The proof of Theorem \ref{ZWTHM90}.}
As a direct consequence of Corollary \ref{AMZW01}, we obtain the following result.
\begin{prop}\label{PRO30}
For $R\in(0,\frac{1}{2}]$, we derive either $\omega\leq 2^{\kappa_{\ast}}AR$, or
\begin{align}\label{N90}
\mathop{osc}\limits_{Q(R/4,m(R/4)^{p})}u\leq\eta_{\ast}\omega,\quad\eta_{\ast}=1-2^{-(\kappa_{\ast}+1)},
\end{align}
with \(\kappa_\ast\) and \(A\) as given in Corollary \ref{AMZW01}.
\end{prop}
\begin{proof}
We begin by distinguishing the following two alternatives:
\begin{align}\label{WAMQ001}
|B_{R/2}\cap\{u(\cdot,-m(R/2)^{p})>\mu^{+}-2^{-1}\omega\}|\leq\frac{1}{2}|B_{R/2}|,
\end{align}
and
\begin{align}\label{WAMQ002}
|B_{R/2}\cap\{u(\cdot,-m(R/2)^{p})<\mu^{-}+2^{-1}\omega\}|\leq\frac{1}{2}|B_{R/2}|.
\end{align}
Hence, it follows from Corollary \ref{AMZW01} that when $\omega>2^{\kappa_{\ast}}AR$,
\begin{align*}
\sup\limits_{Q(R/4,m(R/4)^{p})}u\leq\mu^{+}-2^{-(\kappa_{\ast}+1)}\omega,\quad\text{if \eqref{WAMQ001} holds,}
\end{align*}
and
\begin{align*}
\inf\limits_{Q(R/4,m(R/4)^{p})}u\geq\mu^{-}+2^{-(\kappa_{\ast}+1)}\omega,\quad\text{if \eqref{WAMQ002} holds.}
\end{align*}
Therefore, regardless of which case occurs, we have
\begin{align*}
\mathop{osc}\limits_{Q(R/4,m(R/4)^{p})}u\leq\big(1-2^{-(k_{\ast}+1)}\big)\omega.
\end{align*}
This concludes the proof.

\end{proof}

Using Proposition \ref{PRO30}, we construct below a nested family of cylinders with a common vertex and decreasing radii, such that the essential oscillation of the solution over these cylinders vanishes in the limit as the radii tend to zero.

For $R\in(0,\frac{1}{8}]$, denote
\begin{align}\label{K99}
A_{\ast}:=2^{\kappa_{\ast}}A=2^{\kappa_{\ast}}\mathcal{M}_{\ast},\quad \omega_{0}:=\max\{\omega,A_{\ast}R\},
\end{align}
and, for $k\geq0$,
\begin{align}\label{K98}
R_{k}:=A_{\ast}^{-k}R,\quad \omega_{k+1}:=\max\big\{\eta_{\ast}\omega_{k},A_{\ast}R_{k}\big\},\quad \tilde{a}_{k}:=\Big(\frac{\omega_{k}}{A_{\ast}}\Big)^{2-p}.
\end{align}
Since $p>2$, we have from \eqref{K99}--\eqref{K98} that
\begin{align*}
\tilde{a}_{k+1}R_{k+1}^{p}=&\left(\frac{\omega_{k+1}}{A_{\ast}}\right)^{2-p}\frac{R_{k}^{p}}{A_{\ast}^{p}}\leq\frac{\tilde{a}_{k}R_{k}^{p}}{\eta_{\ast}^{p-2}A_{\ast}^{p}}< \tilde{a}_{k}R_{k}^{p}.
\end{align*}
This shows that $Q(R_{k+1},\tilde{a}_{k+1}R_{k+1}^{p})\subset Q(R_{k},\tilde{a}_{k}R_{k}^{p})$. 
\begin{lemma}\label{LEM83}
For $k\geq0$, we derive
\begin{align}\label{OSC98}
\mathop{osc}\limits_{Q(R_{k},\tilde{a}_{k}R_{k}^{p})}u\leq\omega_{k}.
\end{align}

\end{lemma}

\begin{proof}
By \eqref{K98}, we have $\tilde{a}_{0}\leq a_{0}$, and therefore \eqref{OSC98} holds for $k=0$. Suppose now that \eqref{OSC98} is valid for some $k=i$, where $i\geq1$.\ We prove that the same conclusion holds for $k=i+1$. Using the induction hypothesis $\mathop{osc}\limits_{Q(R_{i},\tilde{a}_{i}R_{i}^{p})}u\leq\omega_{i}$,
and arguing as in Proposition \ref{PRO30}, we deduce
\begin{align}\label{AEM81}
\mathop{osc}\limits_{Q(R_{i}/4,m_{i}(R_{i}/4)^{p})}u\leq\max\{\eta_{\ast}\omega_{i},A_{\ast}R_{i}\}=\omega_{i+1},\quad m_{i}:=\Big(\frac{\omega_{i}}{A_{\ast}}\Big)^{2-p}.
\end{align}
Taking into account that $\omega_{i+1}\geq\eta_{\ast}\omega_{i}$ and $p>2$, we infer
\begin{align*}
	m_{i}\Big(\frac{R_{i}}{4}\Big)^{p}=&\Big(\frac{\omega_{i}}{A_{\ast}}\Big)^{2-p}\Big(\frac{R_{i}}{4}\Big)^{p}\geq\Big(\frac{\omega_{i+1}}{A_{\ast}}\Big)^{2-p}\frac{\eta_{\ast}^{p-2}R_{i}^{p}}{4^{p}}\notag\\
	\geq& \tilde{a}_{i+1}R_{i+1}^{p}\frac{\eta_{\ast}^{p-2}A_{\ast}^{p}}{4^{p}}\geq \tilde{a}_{i+1}R_{i+1}^{p+\vartheta},
\end{align*}
which, in conjunction with \eqref{AEM81}, shows that
\begin{align*}
\mathop{osc}\limits_{Q(R_{i+1},\tilde{a}_{i+1}R_{i+1}^{p})}u\leq\omega_{i+1}.
\end{align*}
This completes the proof.
\end{proof}

Write
\begin{align}\label{AH29}
	\varepsilon_{\ast}=-\frac{\ln\eta_{\ast}}{\ln A_{\ast}}\in(0,1),
\end{align}
where $\eta_{\ast}$ and $A_{\ast}$ are specified in \eqref{N90} and \eqref{K99}, respectively. Based on Lemma \ref{LEM83}, a slight modification of the proof of Lemma 5.6 in \cite{HLZ2026} gives the following oscillation estimate with an explicit decay rate.
\begin{lemma}\label{PRO90}
Let $p>2$. For any $0<\rho\leq R\leq\frac{1}{8}$, $x_{0}\in \overline{B}_{\frac{1}{2}}$, and $t_{0}\in[-\frac{1}{2},0]$, we derive
\begin{align*}
	\mathop{osc}\limits_{[(x_{0},t_{0})+Q(\rho,\tilde{a}_{0}\rho^{p})]}u\leq \max\big\{\eta_{\ast}^{-1}\omega_{0},A_{\ast}^{1+2\varepsilon_{\ast}}R^{\varepsilon_{\ast}}\big\}\left(\frac{\rho}{R}\right)^{\varepsilon_{\ast}},
\end{align*}
where $\omega_{0},\,\tilde{a}_{0}$ and $\varepsilon_{\ast}$ are, respectively, defined in \eqref{K99}--\eqref{K98} and \eqref{AH29}.

\end{lemma}

We are now ready to prove Theorem \ref{ZWTHM90}.
\begin{proof}[Proof of Theorem \ref{ZWTHM90}]
{\bf Step 1.} Using the fact of $p>2$, we deduce from \eqref{M02} and \eqref{K99} that for any $0<R\leq\frac{1}{8}$,
\begin{align}\label{AF09}
\tilde{a}_{0}=\left(\frac{A_{\ast}}{\max\{\omega,A_{\ast}R\}}\right)^{p-2}\geq \left(\frac{A_{\ast}}{\max\{\mathcal{M}_{\ast},A_{\ast}/2\}}\right)^{p-2}\geq1.
\end{align}
As a consequence of \eqref{AF09}, Lemmas \ref{LEMZ01} and \ref{PRO90}, there exists a constant $\alpha\in(0,1)$, depending only on the prescribed data, such that for all
$t_{0}\in[-\frac{1}{2},0]$, $x_{0}\in\overline{B}_{\frac{1}{2}}$, and $\rho\in(0,\frac{1}{8}]$,
\begin{align*}
\mathop{osc}\limits_{[(x_{0},t_{0})+Q(\rho,\rho^{p})]}u\leq C\rho^{\alpha}.
\end{align*}
Combining this with the time-Lipschitz estimate in \eqref{AH01}, we obtain that for any $t_{0}\in[-\frac{1}{2},0]$, $x_{0}\in\overline{B}_{\frac{1}{2}}$, and
$(x,t)\in B_{\frac{1}{8}}(x_{0})\times(-\frac{3}{4},t_{0}]$,

$(1)$ when $|t-t_{0}|\leq8^{-p}$, 
\begin{align}\label{AH32}
&|u(x,t)-u(x_{0},t_{0})|\leq|u(x,t)-u(x,t_{0})|+|u(x,t_{0})-u(x_{0},t_{0})|\notag\\
&\leq\bigg(\frac{\mathcal{Q}}{(p-2)(t+1)}+\|f\|_{L^{\infty}(B_{1})}\bigg)(t_{0}-t)+C|x-x_{0}|^{\alpha}\notag\\
&\leq C\big(|x-x_{0}|^{\alpha}+|t-t_{0}|\big);
\end{align}

$(2)$ when $|t-t_{0}|>8^{-p}$, we can find an increasing sequence $\{t_{i}\}_{i=1}^{N}$ with $1\leq N\leq [3\cdot2^{p-2}]+1$, such that $t<\bar{t}_{1}\leq\cdots\leq \bar{t}_{N}<t_{0}$,
\begin{align}\label{AH323}
	&|u(x,t)-u(x_{0},t_{0})|\notag\\
	&\leq|u(x,t)-u(x,\bar{t}_{1})|+|u(x,\bar{t}_{1})-u(x,t_{0})|+|u(x,t_{0})-u(x_{0},t_{0})|\notag\\
	&\leq\frac{\mathcal{Q}}{p-2}\bigg(\frac{\bar{t}_{1}-t}{t+1}+\frac{t_{0}-\bar{t}_{1}}{\bar{t}_{1}+1}\bigg)+\|f\|_{L^{\infty}(B_{1})}(t_{0}-t)+C|x-x_{0}|^{\alpha}\notag\\
	&\leq C\big(|x-x_{0}|^{\alpha}+|t-t_{0}|\big),\quad\text{if}\;N=1,
\end{align}
and
\begin{align}\label{AH33}
	&|u(x,t)-u(x_{0},t_{0})|\notag\\
	&\leq|u(x,t)-u(x,\bar{t}_{1})|+\sum^{N-1}_{i=1}|u(x,\bar{t}_{i})-u(x,\bar{t}_{i+1})|\notag\\
	&\quad+|u(x,\bar{t}_{N})-u(x,t_{0})|+|u(x,t_{0})-u(x_{0},t_{0})|\notag\\
	&\leq \frac{\mathcal{Q}}{p-2}\bigg(\frac{\bar{t}_{1}-t}{t+1}+\sum^{N-1}_{i=1}\frac{\bar{t}_{i+1}-\bar{t}_{i}}{\bar{t}_{i}+1}+\frac{t_{0}-\bar{t}_{N}}{\bar{t}_{N}+1}\bigg)+\|f\|_{L^{\infty}(B_{1})}(t_{0}-t)+C|x-x_{0}|^{\alpha}\notag\\
	&\leq C\big(|x-x_{0}|^{\alpha}+|t-t_{0}|\big),\quad \text{if}\;N\geq2,
\end{align}
where $\mathcal{Q}$ is defined in \eqref{DK02}.

{\bf Step 2.} Fix $x_{0}\in\overline{B}_{\frac{1}{2}}.$ For any given $(x,t),(\hat{x},\hat{t})\in B_{\frac{1}{8}}(x_{0})\times(-\frac{1}{2},0],$ we may suppose, without loss of generality, that $|\hat{x}-x_{0}|\geq|x-x_{0}|$. Let $R=|\hat{x}-x_{0}|$.  For each $(y,s)\in Q(1/(2R),1/R^{p})$, we define the rescaled functions by
\begin{align*}
u_{R}(y,s)=u(Ry+x_{0},R^{p}s),\quad f_{R}(y)=R^{p}f(Ry+x_{0}).
\end{align*}
Hence $u_{R}$ satisfies
\begin{align*}
\partial_{s}u_{R}-\mathrm{div}(|\nabla u_{R}|^{p-2}\nabla u_{R})=f_{R}(y),\quad \mathrm{in}\;Q(1/(2R),1/R^{p}).
\end{align*}
Proceeding as in \eqref{AH32}, \eqref{AH323} and \eqref{AH33} with minor modifications, we find constants $0<\beta<\alpha$ and $C>1$, depending solely on the data, satisfying that for any given $\bar{y}\in\partial B_{1}$ and $\bar{s}\in(-2^{-1}R^{-p},0]$,
\begin{align}\label{WAQA001}
|u_{R}(y,s)-u_{R}(\bar{y},\bar{s})|\leq C\big(|y-\bar{y}|^{\beta}+|s-\bar{s}|\big),
\end{align}
for every $(y,s)$ with $|y-\bar{y}|+|s-\bar{s}|^{\frac{1}{p}}\leq\frac{1}{8}$. 

From \eqref{AH01}, it follows that
\begin{align}\label{QT01}
|u(x,t)-u(\hat{x},\hat{t})|\leq&|u(x,t)-u(x,\hat{t})|+|u(x,\hat{t})-u(\hat{x},\hat{t})|\notag\\
\leq& \bigg(\frac{2\mathcal{Q}}{p-2}+\|f\|_{L^{\infty}(B_{1})}\bigg)|t-\hat{t}|+|u(x,\hat{t})-u(\hat{x},\hat{t})|.
\end{align}
It remains to estimate the second term in \eqref{QT01}.\ Set $c\geq2$.\ If $|x-\hat{x}|\leq R^{c}$, we obtain from \eqref{WAQA001} that
\begin{align*}
|u(x,\hat{t})-u(\hat{x},\hat{t})|=&\left|u_{R}((x-x_{0})/R,\hat{t}/R^{p})-u_{R}((\hat{x}-x_{0})/R,\hat{t}/R^{p})\right|\notag\\
\leq& C|(x-\hat{x})/R|^{\beta}\leq C|x-\hat{x}|^{\frac{(c-1)\beta}{c}}.
\end{align*}
In the complementary case of $|x-\hat{x}|>R^{c}$, \eqref{AH32} gives
\begin{align*}
|u(x,\hat{t})-u(\hat{x},\hat{t})|\leq&|u(x,\hat{t})-u(x_{0},\hat{t})|+|u(x_{0},\hat{t})-u(\hat{x},\hat{t})|\notag\\
\leq& C(|x-x_{0}|^{\alpha}+R^{\alpha})\leq C|x-\hat{x}|^{\frac{\alpha}{c}}.
\end{align*}
Then fixing $c=1+\frac{\alpha}{\beta}$, it follows that for every $(x,t),(\hat{x},\hat{t})\in B_{\frac{1}{8}}(x_{0})\times(-\frac{1}{2},0],$
\begin{align*}
|u(x,t)-u(\hat{x},\hat{t})|\leq C\big(|x-\hat{x}|^{\frac{\alpha\beta}{\alpha+\beta}}+|t-\hat{t}|\big),
\end{align*}
which, together with finite covering argument, gives that for any $(x,t),(y,s)\in B_{\frac{5}{8}}\times(-\frac{1}{2},0]$,
\begin{align*}
|u(x,t)-u(y,s)|\leq &C\big(|x-y|^{\frac{\alpha\beta}{\alpha+\beta}}+|t-s|\big).	
\end{align*}	
The proof is thereby complete.

\end{proof}

\section{$W_{x,t}^{1,\infty}$ estimates for weak solutions}\label{SEC007}
Based on the temporal Lipschitz estimate established in Theorem \ref{thm10}, we now prove space-time $W_{x,t}^{1,\infty}$ estimates for the solution after an arbitrary positive waiting time.
\begin{proof}[Proof of Theorem \ref{ZWTHM96}]
Combining Theorem 4.1 in \cite{HLZ2026}, Lemma \ref{LEMZ01}, Theorem \ref{thm10}, Remarks \ref{RE09} and \ref{RE10}, we obtain that for any fixed $0<\delta\ll 1$ and $-1+\delta\leq s<t\leq0$, 
\begin{align*}
	\|u\|_{L^{\infty}(Q(1,1))}\leq\|f\|_{L^{\infty}(B_{1})}+\|\phi\|_{L^{\infty}(B_{1})},
\end{align*}
and		
\begin{align}\label{AE03}
	\|\partial_{t}u\|_{L^{\infty}(B_{1}\times(-1+\delta,0))}\leq
	\begin{cases}
		\frac{\mathcal{Q}}{\delta(2-p)},&1<p<2,\\
		\frac{\mathcal{Q}}{\delta(p-2)}+\|f\|_{L^{\infty}(B_{1})},&p>2,	
	\end{cases}		
\end{align}		
where $\mathcal{Q}$ is given by \eqref{DK02}. For simplicity, denote
\begin{align}\label{G01}
g(x,t):=f(x)-\partial_{t}u(x,t)\in L^{\infty}(B_{1}\times(-1+\delta,0)).	
\end{align}	

 We next divide the proof into two parts.
	
\noindent\textbf{Part 1.} Proof of \eqref{Neu01}.\ Observe from \eqref{AE03} that for almost every $t\in(-1+\delta,0)$, the solution $u(x,t)$ satisfies
\begin{align*}
	\begin{cases}
		-\mathrm{div}(|\nabla u(x,t)|^{p-2}\nabla u(x,t))=g(x,t),&\mathrm{in}\;B_{1},\\
		|\nabla u(x,t)|^{p-2}\nabla u(x,t)\cdot\nu(x)=0,&\mathrm{on}\;\partial B_{1},
	\end{cases}
\end{align*}
where $g$ is defined by \eqref{G01}.\ Then using Theorem 1.1 in \cite{CM2011} and Remark 2.7 in \cite{CM2014}, we directly obtain that for a.e. $t\in(-1+\delta,0),$
\begin{align*}
\|\nabla u(\cdot,t)\|_{L^{\infty}(B_{1})}\leq& C(n,p)\|g(\cdot,t)\|_{L^{\infty}(B_{1})}^{\frac{1}{p-1}}\notag\\
\leq&C(n,p) \bigg(\frac{\|f\|_{L^{\infty}(B_{1})}+\|\phi\|_{L^{\infty}(B_{1})}}{\delta|p-2|}+\|f\|_{L^{\infty}(B_{1})}\bigg)^{\frac{1}{p-1}}.
\end{align*}		
Combining this with \eqref{AE03}, we prove \eqref{Neu01}.

\noindent \textbf{Part 2.} Proof of \eqref{Dir02}. From \eqref{AE03}, we see that for a.e. $t\in(-1+\delta,0)$, the function $u(x,t)$ solves
	\begin{align}\label{ELLIPTIC01}
		\begin{cases}
			-\mathrm{div}(|\nabla u(x,t)|^{p-2}\nabla u(x,t))=g(x,t),&\mathrm{in}\;B_{1},\\
			u(x,t)=\phi(x),&\mathrm{on}\;\partial B_{1},
		\end{cases}
	\end{align}
where $g$ is given by \eqref{G01}.\ Claim that for any constant $c\in\mathbb{R}$, $x_{0}\in B_{1}$ and $0<r\leq\frac{1-|x_{0}|}{2}$, we have the following interior gradient estimate:
	\begin{align}\label{LOCALBOUND01}
		\|\nabla u\|_{L^{\infty}(B_{r}(x_{0}))}
		\leq C\left(\frac{\|u-c\|_{L^{\infty}(B_{2r}(x_{0}))}}{r}+\big(r\|g\|_{L^{\infty}(B_{2r}(x_{0}))}\big)^{1/(p-1)}\right),
	\end{align}
where $C=C(n,p).$ Take $\zeta\in C_{0}^{\infty}(B_{2r}(x_{0}))$ with $0\leq\zeta\leq1$, $\zeta=1$ on $B_{3r/2}(x_{0})$, and $|\nabla\zeta|\leq 4/r$. Testing \eqref{ELLIPTIC01} with $(u-c)\zeta^{p}$, we obtain
	\begin{align*}
		\int_{B_{2r}(x_{0})}|\nabla u|^{p}\zeta^{p}dx\leq&p\int_{B_{2r}(x_{0})}|u-c|\zeta^{p-1}|\nabla u|^{p-1}|\nabla\zeta|+\int_{B_{2r}(x_{0})} |g||u-c|\zeta^{p}dx,
	\end{align*}
which, combined with Young's inequality, gives that
	\begin{align}\label{CACCI02}
		\int_{B_{3r/2}(x_{0})}|\nabla u|^{p}dx
		\leq C(n,p)\big(\|u-c\|_{L^{\infty}}^{p}r^{n-p}+\|u-c\|_{L^{\infty}}\|g\|_{L^{\infty}}r^{n}\big).
	\end{align}
Therefore, applying Theorem 1.1 in \cite{DM201002} with \eqref{CACCI02}, we deduce from Young's inequality that
\begin{align*}
\|\nabla u\|_{L^{\infty}(B_{r}(x_{0}))}\leq& C\left(\frac{\|u-c\|_{L^{\infty}}}{r}+\big(\|u-c\|_{L^{\infty}}\|g\|_{L^{\infty}}\big)^{\frac{1}{p}}+\big(r\|g\|_{L^{\infty}}\big)^{1/(p-1)}\right)\notag\\
\leq&C(n,p)\left(\frac{\|u-c\|_{L^{\infty}}}{r}+\big(r\|g\|_{L^{\infty}}\big)^{1/(p-1)}\right),
\end{align*}
which implies that \eqref{LOCALBOUND01} holds.

We next construct the required boundary barriers. For simplicity, write
\begin{align*}
L_{1}=&\|f\|_{L^{\infty}(B_{1})},\quad L_{2}=\|\phi\|_{L^{\infty}(B_{1})},\quad L_{3}=\|\nabla\phi\|_{L^{\infty}(B_{1})},\\
L_{4}=&\mathop{\mathrm{ess\,sup}}_{x\in B_{1}}\mathop{\sup}_{|\xi|=1}|\nabla^{2}\phi(x)\xi|,\quad 
L_{5}=\begin{cases}
	\frac{\mathcal{Q}}{\delta(2-p)},&1<p<2,\\
	\frac{\mathcal{Q}}{\delta(p-2)}+\|f\|_{L^{\infty}(B_{1})},&p>2,
\end{cases}		
\end{align*}	
where $\mathcal{Q}$ is specified by \eqref{DK02}. Define $c_{n,p}=(n+p-2)\min\limits_{1/2\leq\sigma\leq5/2}\sigma^{p-2}$, and
\begin{align}\label{CON06}
 \mathcal{L}=\max\bigg\{2L_{3},\,L_{4},\,\frac{4}{3}(L_{1}+2L_{2}),\,\Big(\frac{L_{1}+L_{5}}{c_{n,p}}\Big)^{\frac{1}{p-1}}\bigg\}.
\end{align}
If $\mathcal{L}=0$, then $u=0$ and thus Theorem  \ref{ZWTHM96} is proved.\ Assume henceforth that $\mathcal{L}>0$ and define
	\begin{align*}
		\bar{v}(x)=\phi(x)+\mathcal{L}(1-|x|^{2}),\quad \underline{v}(x)=\phi(x)-\mathcal{L}(1-|x|^{2}).
	\end{align*}
From \eqref{CON06}, we have 
\begin{align}\label{GRADIENT01}
\begin{cases}
\frac{\mathcal{L}}{2}\leq |\nabla \overline{v}(x)|=|\nabla\phi(x)-2\mathcal{L}x|\leq\frac{5\mathcal{L}}{2},\\
\frac{\mathcal{L}}{2}\leq |\nabla \underline{v}(x)|=|\nabla\phi(x)+2\mathcal{L}x|\leq\frac{5\mathcal{L}}{2},
\end{cases}\quad\text{for }x\in B_{1}\setminus\overline{B}_{\frac{1}{2}},
\end{align}
and
\begin{align}\label{HESSIAN02}
\begin{cases}	
	\nabla^{2}\overline{v}(x)=\nabla^{2}\phi(x)-2\mathcal{L}I\leq-\mathcal{L}I,\\
 \nabla^{2}\underline{v}(x)=\nabla^{2}\phi(x)+2\mathcal{L}I\geq \mathcal{L}I,
 \end{cases}\quad\text{for a.e. }x\in B_{1}\setminus\overline{B}_{\frac{1}{2}},
\end{align}
where $I$ represents the $n\times n$ identity matrix.
	
For a function $v$ with nonzero gradient, write $e=\nabla v/|\nabla v|$. A direct calculation shows that
\begin{align}\label{NONDIVERGENCE01}
\mathrm{div}(|\nabla v|^{p-2}\nabla v)=|\nabla v|^{p-2}\mathrm{tr}\left[\big(I+(p-2)e\otimes e\big)\nabla^{2}v\right],
\end{align}
where $\mathrm{tr}$ denotes the trace of a matrix.\ The matrix $I+(p-2)e\otimes e$ is positive definite for every $p>1$: its eigenvalues are $1$ with multiplicity $n-1$ and $p-1$ in the direction $e$, which implies that its trace is $n+p-2$. Therefore, a consequence of \eqref{CON06}, \eqref{GRADIENT01}, \eqref{HESSIAN02} and \eqref{NONDIVERGENCE01} gives that for a.e. $(x,t)\in B_{1}\times(-1+\delta,0)$,
\begin{align}\label{BARRIERSIGNS01}
\begin{cases}		
		-\mathrm{div}(|\nabla \overline{v}|^{p-2}\nabla \overline{v})\geq c_{n,p}\mathcal{L}^{p-1}\geq L_{2}+L_{5}\geq g(x,t),\\
		-\mathrm{div}(|\nabla \underline{v}|^{p-2}\nabla \underline{v})\leq-c_{n,p}\mathcal{L}^{p-1}\leq-(L_{2}+L_{5})\leq g(x,t).
\end{cases}
\end{align}
On the outer boundary, we have
\begin{align}\label{OUTER01}
	\overline{v}=\underline{v}=\phi=u,\quad\mathrm{on}\;\partial B_{1}.
\end{align}
From \eqref{DK01} and \eqref{CON06}, we obtain
\begin{align}\label{INNER01}
\begin{cases}	
	\overline{v}=\phi+\frac{3}{4}\mathcal{L}\geq-L_{2}+\frac{3}{4}\mathcal{L}\geq L_{1}+L_{2}\geq u,\\
	\underline{v}=\phi-\frac{3}{4}\mathcal{L}\leq L_{2}-\frac{3}{4}\mathcal{L}\leq-(L_{1}+L_{2})\leq u,
\end{cases}	\quad \text{on } \partial B_{\frac{1}{2}}.
\end{align}
In light of \eqref{BARRIERSIGNS01}, \eqref{OUTER01} and \eqref{INNER01}, it follows from the comparison principle that
\begin{align*}
	\underline{v}(x)\leq u(x,t)\leq \overline{v},\quad\text{for a.e. }(x,t)\in B_{1}\times(-1+\delta,0),
\end{align*}
which gives that for almost every $(x,t)\in B_{1}\times(-1+\delta,0)$,
\begin{align}\label{DM01}
	|u(x,t)-\phi(x)|\leq \mathcal{L}(1-|x|^{2})\leq2\mathcal{L}\mathrm{dist}(x,\partial B_{1}).
\end{align}

To establish the boundary estimate, it remains to remove the distance factor arising in the interior estimate of \eqref{LOCALBOUND01}. Let
\begin{align*}
x_{\ast}\in B_{1}\setminus B_{\frac{3}{4}},\quad d_{\ast}=1-|x_{\ast}|,\quad r=\frac{d_{\ast}}4.
\end{align*}
Then for any $x\in B_{2r}(x_{\ast})\subset B_{1}\setminus\overline{B}_{\frac{1}{2}}$, we have
\begin{align*}
	\mathrm{dist}(x,\partial B_{1})\leq d_{\ast}+|x-x_{\ast}|\leq\frac{3}{2}d_{\ast}.
\end{align*}
Using \eqref{CON06} and \eqref{DM01}, we find
\begin{align*}
	|u(x,t)-\phi(x_{\ast})|\leq& |u(x,t)-\phi(x)|+|\phi(x)-\phi(x_{\ast})|\leq3\mathcal{L}d_{\ast}+\frac{1}{2}L_{3}d_{\ast}\notag\\
	\leq&4\mathcal{L}d_{\ast},\quad\text{for a.e. }(x,t)\in B_{2r}(x_{\ast})\times(-1+\delta,0).
\end{align*}
Applying \eqref{LOCALBOUND01} with $x_{0}=x_{\ast}$, $r=\frac{d_{\ast}}{4}$ and $c=\phi(x_{\ast})$, it follows from \eqref{CON06} that
\begin{align}\label{AE01}
	\|\nabla u\|_{L^{\infty}(B_{r}(x_{\ast}))}\leq C(n,p)\Big(\frac{4\mathcal{L}d_{\ast}}{r}+(r(L_{1}+L_{5}))^{\frac{1}{p-1}}\Big)\leq C(n,p)\mathcal{L}.
\end{align}
For $x_{\ast}\in B_{\frac{3}{4}}$, we can choose a fixed radius $r=\frac{1}{16}$ independent of the point such that $B_{2r_{\ast}}(x_{\ast})\subset B_{1}$. Then applying \eqref{LOCALBOUND01} with $x_{0}=x_{\ast}$, $r=\frac{1}{16}$ and $c=0$, we have from \eqref{CON06} that
\begin{align}\label{AE02}
	\|\nabla u\|_{L^{\infty}(B_{r_{\ast}}(x_{\ast}))}\leq C(n,p)\Big(L_{1}+L_{2}+(L_{1}+L_{5})^{\frac{1}{p-1}}\Big)\leq C(n,p)\mathcal{L}.
\end{align}
Therefore, a combination of \eqref{AE01}--\eqref{AE02} gives that
\begin{align*}
	\|\nabla u(\cdot,t)\|_{L^{\infty}(B_{1})}\leq C(n,p)\mathcal{L},\quad \text{for a.e. }t\in(-1+\delta,0),
\end{align*}
which, in conjunction with \eqref{AE03}, implies that \eqref{Dir02} holds.

\end{proof}

%
%%
%\noindent{\bf{\large Conflict of interest.}} The authors declare that they have no conflict of interest.
%\noindent{\bf{\large Data Availability Statement.}} The data used to support the findings of this study are available from the corresponding author upon request.

\noindent{\bf{\large Acknowledgements.}} X. Hao was partially supported by Science and Technology Project of Hebei Education Department (No.\ QN2024074) and Hebei Natural Science Foundation (No.\ A2026205019).\ Z.W. Zhao was partially supported by NSFC (No.\ 12501254).

%\noindent{\bf{\large Statements.}}

%\noindent{\bf{\large Acknowledgements.}}

\end{document}